\documentclass[10pt]{article}

\usepackage{amssymb}   
\usepackage{amsthm}    
\usepackage{amsmath}   
\usepackage{stmaryrd}  
\usepackage{titletoc}  
\usepackage{mathrsfs}  
\usepackage{graphicx}
\usepackage{xcolor}
\usepackage[toc,page,title,titletoc,header]{appendix}
\usepackage{dsfont}
\usepackage{extarrows}
\usepackage{geometry}
\usepackage{bm}

\newlength{\defbaselineskip}
\newcommand{\setlinespacing}[1]%
           {\setlength{\baselineskip}{#1 \defbaselineskip}}

\theoremstyle{plain}

\makeatletter\@addtoreset{equation}{section} \makeatother
 \allowdisplaybreaks

\newtheorem{theorem}{Theorem}[section]
\newtheorem{lemma}[theorem]{Lemma}
\newtheorem{proposition}[theorem]{Proposition}
\newtheorem{corollary}[theorem]{Corollary}

\theoremstyle{definition}
\newtheorem{definition}[theorem]{Definition}

\theoremstyle{remark}
\newtheorem{remark}[theorem]{Remark}

\numberwithin{equation}{section}

\begin{document}
\title{\vspace{-12.6mm}Extended Mean Field Games with Singular Controls
}

\author{Guanxing Fu\footnote{Department of Applied Mathematics, The Hong Kong Polytechnic University,
         Hung Hom, Kowloon, Hong Kong; email: guanxing.fu@polyu.edu.hk\newline The author would like to thank the anonymous Associate Editor and two anonymous reviewers, for their great and valuable comments and suggestions. All of them improve the quality of the paper. Moreover, the author would like to thank all people who made great efforts to fund him during the tough period in Covid-19 when this work was completed, including Min Dai, Ulrich Horst, Defeng Sun and Chao Zhou.
 }}
%
\date{}
\maketitle

\vspace{-10.6mm}
\begin{abstract}
This paper establishes the existence of equilibria result of a class of mean field games with singular controls. The interaction takes place through both states and controls. A relaxed solution approach is used. To circumvent the tightness issue, we prove the existence of equilibria by first considering the corresponding mean field games with continuous controls instead of singular controls and then taking approximation.
\end{abstract}

\vspace{-2.6mm}


\vspace{-6.6mm}
\section{Introduction}

\vspace{-5.6mm}
Mean field games (MFGs), introduced by \cite{HMC-2006} and \cite{LL-2007}, provide a powerful tool to study approximate Nash equilibria in symmetric large population games, where the interaction only takes place in the empirical distribution of all players' states or strategies. {The methodology is to first approximate the empirical distribution by an exogenously given distribution, and to consider the standard optimization problem of a representative player, and then to search for the fixed point such that the distribution of the representative player’s optimal state or strategy is identical to the given one}; for details, we refer to the monograph \cite{CD-2018ab}. By now, most works study MFGs with absolutely continuous strategies (regular controls) and there are limited results on MFGs with singular controls. Among them,
MFGs with singular controls were first studied in \cite{FH-2017}, where the existence of equilibria result for a class of MFGs was established using a relaxed solution approach. By studying the quasi-variational inequalities, \cite{Cao2020} and \cite{GX-2018} characterized the equilibria of MFGs with singular controls in infinite horizons. {\cite{Hu-Oksendal-Sulem-2017} examined a class of mean field type games with singular controls by maximum principle approach.} The recent work \cite{Campi2020} considered an MFG with finite fuel arising in goodwill problems. By using the connection with optimal stopping problems, \cite{Campi2020} obtained the equilibrium iteratively. 
 In all these papers, the interaction takes place only through states. In our paper, motivated by the optimal portfolio liquidation problem and optimal exploitation of exhaustible resource (see Section \ref{sec:motivation}) we introduce a novel class of \textit{extended} MFGs with singular controls where the interaction takes place through both states and actions
\begin{equation}\label{general-MFG}
\left\{ \begin{split}
1. &~\textrm{For fixed probability measures }\mu:=(\mu^{(1)},\mu^{(2)},\mu^{(3)},\mu^{{4}},\mu^{(5)})\textrm{ in some suitable space,}\\
&~\textrm{solve the optimization problem}: \textrm{minimize }J(u,Z^{(1)},Z^{(2)};\mu) \textrm{ such that}\\
&~dX^{(i)}_t=b^{(i)}(t,X^{(i)}_t,\mu^{(i)}_t)\,dt+d(\kappa^{(i)}\overline\mu^{(i)}_t+\eta^{(i)}{Z^{(i)}_t})+\sigma^{(i)}(t)dW^{(i)}_t,~i=1,2, \textrm{ and }\\
&~dX^{(3)}_t=b^{(3)}(t,X^{(3)}_t,u_t)\,dt+\alpha^{(1)} dZ^{(1)}_t-\alpha^{(2)}dZ^{(2)}_t+l^{}(t,u_t)\widetilde N^{}(dt).\\
2.&~\textrm{Search for the fixed point }\mu=(\mathcal L (Z^{(1)}),\mathcal L(Z^{(2)}),\mathcal L(X^{(1)}),\mathcal L(X^{(2)}),\mathcal L(X^{(3)})),\\
&~\textrm{where }Z^{(1)},~Z^{(2)}\textrm{ and }X^{(1)},~X^{(2)},~X^{(3)}\textrm{ are the optimal controls and states from 1},
\end{split}\right.
\end{equation}
where the cost functional follows

\vspace{-5.6mm}
\begin{equation}\label{cost-after-1.1}
	\begin{split}
		&~J(u,Z^{(1)},Z^{(2)};\mu)\\
		=&~\mathbb E\left[\int_0^T\sum_{i=1}^2 h(X^{(i)}_s)\cdot \,d(\kappa^{(i)}\overline\mu^{(i)}_s+\eta^{(i)}{Z_s^{(i)}})^c+\sum_{i=1}^2\sum_{0\leq t\leq T}\sum_{j=1}^d\int_0^{\Delta X^{(i)}_{j,t}}h_j(X^{(i)}_{j,t-}+x)\,dx\right.\\
		&~\qquad\qquad\left.+\int_0^Tf(t,X^{(1)}_t,X^{(2)}_t,X^{(3)}_t,\mu_t,u_t)\,dt+g(X^{(1)}_T,X^{(2)}_T,X^{(3)}_T,\mu_T)\right].
	\end{split}
\end{equation}

\vspace{-4.6mm}
In \eqref{general-MFG} and \eqref{cost-after-1.1}, $\widetilde N$ is a compensated Poisson process with intensity function $\lambda$, $W^{(1)}$ and $W^{(2)}$ are two Brownian motions defined on some probability space, $u$ is the regular control, $Z^{(1)}$ and $Z^{(2)}$ are singular controls, whose trajectories are c\`adl\`ag and non-decreasing, $X^c$ stands for the continuous part of $X$, $\Delta X_t$ is the jump of $X$ at $t$, $\overline\mu^{(i)}$ is the first moment of $\mu^{(i)}$, $i=1,2$ and $\mathcal L(\cdot)$ is the law of $\cdot$.

\vspace{-1.6mm}
 Differently from standard MFGs, where the interaction is only through the states, the interaction in \eqref{general-MFG} takes place not only through the states $X^{(1)}$, $X^{(2)}$ and $X^{(3)}$, but also through the singular controls $Z^{(1)}$ and $Z^{(2)}$. Extended MFGs were analyzed in \cite{CL-2015,Gomes2013} by using probabilistic and analytical approaches, respectively.
 We apply the relaxed solution method (probabilistic compactification method)\footnote{The relaxed solution method we apply is in the sense of \cite{L-2015}, which is different from \cite{BDT-2018}.} to establish the existence of equilibria result. 
The relaxed solution method was first applied to MFGs in \cite{L-2015}. Later it was used to prove existence of equilibria for MFGs with controlled jumps \cite{BCP-2019,BCP-2017}, MFGs with absorptions where the interaction takes place through the empirical distribution of players remaining in the game \cite{Campi2018}, and through both surviving players and past absorptions \cite{Campi2019}, MFGs with common noise \cite{CDL-2016}, MFGs with finite states \cite{Cec:Fis-2017} and MFGs with singular controls \cite{FH-2017}. {The idea is to work with Berge's maximum theorem together with Kakutani-Fan-Glicksberg fixed point theorem; first, establish the closed graph property of the representative player's best response correspondence to a given $\mu$ by the former theorem, and then use the latter theorem to prove the best response correspondence admits a fixed point, which turns out to be a solution to the MFG}; for details, one can refer to the lecture note \cite{Lacker2018}.

\vspace{-1.6mm}
The MFG with Poisson jumps while without singular controls was studied in \cite{BCP-2017}. However, the existence of singular controls makes our problem essentially different from \cite{BCP-2017}. In particular, the Skorokhod $J_1$ topology used in \cite{BCP-2017} does not work for \eqref{general-MFG}. Motivated by \cite{FH-2017} (one dimensional case)  and \cite{FH-2019} (multidimensional case), we work with the Skorokhod $M_1$ topology, {which is weaker than $J_1$ and stronger than the widely used Meyer-Zheng topology (see e.g. \cite{Dianetti2020,Li2017}, where Meyer-Zheng topology was used to approximate controls of finite variations by Lipschitz continuous controls when studying stochastic games/controls of singular type), because (1) the set of bounded monotone functions is compact in the $M_1$ topology but not in the $J_1$ topology; (2) the $M_1$ topology allows for convergence of functions with unmatched jumps; (3) $M_1$ topology is metrizable with explicit metric while the metric for Meyer-Zheng topology is not explicit. Thus, one cannot bound the value of a trajectory at each time point by the corresponding metric. These three properties are essential to prove the existence of equilibria result of \eqref{general-MFG}}. Loosely speaking, there are two $M_1$ topologies, the strong one and the weak one. They coincide  with each other for one dimensional paths, which are the usual objectives in the literature; see \cite{BBF-2019,DIRT-2015,FH-2017,NM-2019} among others. For multidimensional paths the weak $M_1$ topology has an advantage over the strong $M_1$ topology since the oscillation function for weak $M_1$ is always $0$ for monotone paths; see \cite{BK-2015,Cohen-2019,FH-2019}. So in this paper, by $M_1$ topology we always mean the weak $M_1$ topology unless otherwise stated. For the detailed definition and properties of weak $M_1$ topology, we refer to the book \cite[Chapter 12]{Whitt-2002}; see also the recent interesting work \cite[Section 3]{Cohen-2019} for a summary, where Cohen highlighted the advantange of $M_1$ topology over $J_1$ topology by proving the existence of optimal controls for a class of singular control problems with state and control constraint by a simpler proof than \cite{Budhiraja2006}.

\vspace{-1.6mm}
Due to the c\`adl\`ag regularity of singular controls, the relaxed solution method, which might not work for extended MFGs with regular controls as in \cite{CL-2015}, still works for \eqref{general-MFG}. Consequently, the states in \eqref{general-MFG} are allowed to be degenerate. The property of degeneracy is important in applications; see Section \ref{sec:motivation} and \cite{Campi2020}. 
{Although in \cite{FH-2017} we solved the MFGs with singular controls with relaxed solution method, singular controls enter \eqref{general-MFG} in a different way from \cite{FH-2017}. In \cite{FH-2017} singular controls enter the game through the form of $\int_0^\cdot c_s\,dZ_s$, no matter in the state or in the cost. Since $\int_0^\cdot c_s\,dZ_s=\int_0^\cdot c^+_s\,dZ_s-\int_0^\cdot c^-_s\,dZ_s$ and $c^+c^-\equiv 0$, simultaneous jumps never occur. However, $Z^{(1)}$ and $-Z^{(2)}$ in \eqref{general-MFG} may jump at the same as time and in different directions. Such feature does not appear in \cite{FH-2017}.}
One difficulty of our paper comes from the possible simultaneous jumps in different directions. It is well-acknowledged that the Skorokhod space endowed with the $M_1$ topology is not a vector space, in the sense that if $x_n\rightarrow x$ and $y_n\rightarrow y$ in the $M_1$ topology, it is not necessarily true that $x^n+y^n\rightarrow x+y$ in the $M_1$ topology. One possible condition to make it true is $x$ and $y$ do not admit simultaneous jumps in different directions, i.e., $\Delta x_t\Delta y_t\geq 0$. However, this is not our case because of the simultaneous jumps in different directions of singular controls and the Poisson integral, which make it difficult to establish convergence and relative compactness results under the $M_1$ topology. To overcome this difficulty, we follow a two-step strategy: in step 1, instead of considering $Z^{(1)}$ and $Z^{(2)}$ we consider their continuous counterparts $k\int_{\cdot-1/k}^\cdot Z^{(1)}_s\,ds$ and $k\int_{\cdot-1/k}^\cdot Z^{(2)}_s\,ds$. The resulting MFG indexed by $k$ has only one jump process, the Poisson process, and hence it can be analyzed by using the $M_1$ topology. Although in this step the $J_1$ topology works as well, we prefer to proceed with the $M_1$ topology because the approximation in step 2 requires the use of the $M_1$ topology, in which we show the sequence of equilibria indexed by $k$ from step 1 helps construct an equilibrium of \eqref{general-MFG} by approximation. The approximation from step 1 to step 2 holds only under the $M_1$ but not the $J_1$ topology as the $M_1$ topology allows convergence of unmatched jumps. Note that the approximant $k\int_{\cdot-1/k}^\cdot Z_s\,ds$ was also used in \cite{FH-2017}, where we established a relationship between MFGs with singular controls and MFGs with regular controls. {Even so, our paper is not an immediate generalization from \cite{FH-2017}. First, the approximating sequence in \cite{FH-2017} was proved to be relatively compact while in the current paper the approximating sequence associated with $X^{(3)}$ can never be expected to be relatively compact, because of \textit{the simultaneous jumps of singular controls in different directions.} Instead of struggling with the relative compactness issue of $X^{(3)}$, we search for a candidate of Nash equilibrium by construction through nested transformation of probability spaces; see Section \ref{sec:approximation}. Second, }in \cite{FH-2017}, we assumed processes did not admit jumps at the terminal time $T$. In the current paper we drop this assumption by considering a slightly different MFG on a possibly larger horizon in step 1 and assume the coefficients to be trivially extended to this larger horizon in step 2. {Here we should emphasize that the trick of enlargement of the time horizon and trivial extension of the coefficients \textit{does not} reduce the generality of our problem at all. We prove this point by showing the limit in step 2 is supported on the original space and the martingale property is satisfied; see Lemma \ref{lem:stability-martingale-2}. In addition to the simultaneous jumps, we complement the MFG literature by introducing and solving a new MFG with singular controls where the interaction takes place through strategies. The strategic interaction makes the problem with \textit{general} singular controls difficult, even if the interaction only takes place through the first moment. In order to get the relative compactness result with general singular controls, we need a uniform bound of the sequence of laws obtained from Section \ref{proof}. We achieve the goal by doing a fine estimate of the upper bound and the lower bound of the state and the cost; see Lemma  \ref{lem:uniform-bound}.} 

\vspace{-2.1mm}
The remainder is organized as follows: we introduce two motivating examples in Section \ref{sec:motivation}. In Section \ref{sec:model-setup} we introduce the model setup and two main results: the existence of equilibria result of \eqref{general-MFG} with finite fuel, and the existence of equilibria result of \eqref{general-MFG} with general singular controls under additional coercive assumptions of the coefficients.
The proofs  are given in Section \ref{proof} and Section \ref{sec:general-control}, respectively.  \\[-3.66em]

\section{Motivation}\label{sec:motivation} 

 \vspace{-5.6mm}  {Game theory is the study of mathematical models of {\it strategic interactions} among rational players. In this section, we introduce two examples of MFGs with strategic interaction, which motivate our study of the general MFG \eqref{general-MFG}.}
 
  \vspace{-1.6mm}
\textbf{\large{2.1\quad Optimal Portfolio Liquidation}}

\vspace{-2.6mm}
In classic liquidation models, a large trader would like to unwind her open position by submitting market orders in blocked shape into the order book. Due to the limited liquidity, the large orders would move the order book in an unfavorable direction, making the immediate execution costly. However, slow trading may result in high inventory risk due to the market uncertainty.  Thus, the trader needs to make a decision of the trading rate in order to minimize her trading cost (or maximize her net profit). One can refer to \cite{BBF-2019,HN-2014} among others for liquidation with singular controls. Recently, liquidation models beyond single player especially MFGs of optimal liquidation have drawn a lot of attentions; see e.g.  \cite{CL-2016,Casgrain-2018,Casgrain2018a,FGHP-2018,FH-2018,FHX-2020,HJN-2015}, the common nature of which is that the trading price is influenced not only by the individual trader's strategy but also by the aggregation of the competitors' strategies. 

\vspace{-1.6mm}
So far the literature on MFGs of optimal liquidation is focused on absolutely continuous strategies. {However, in e.g. cryptocurrency market, an initial block-shaped execution is often observed, and absolutely continuous strategies are not appropriate to model such phenomenon.}  In the first example, we introduce a model of optimal portfolio liquidation with singular controls, which is a variant of \cite{HN-2014}. Instead of describing the trading price, it is more convenient to consider the spread directly. Following \cite{HN-2014}, we assume the buy spread $X^{(1)}$ and the sell spread $X^{(2)}$ of a representative player follow the dynamics

\vspace{-5.6mm}
\begin{equation}\label{buy-spread}
    X_t^{(i)}=\chi^{(i)}-\int_0^t\rho^{(i)} X^{(i)}_s\,ds+\kappa^{(i)}\nu^{(i)}_t+\eta^{(i)}Z^{(i)}_t+\int_0^t\sigma^{(i)}(s)\,dW^{(i)}_s,\quad i=1,2.
\end{equation}

\vspace{-3.6mm}
Here, $Z^{(1)}$ and $Z^{(2)}$ are the accumulative market buy and sell orders until time $t$, respectively, $\nu^{(1)}$ and $\nu^{(2)}$ are the aggregated (mean-field) market buy and sell orders of competitors. {It reflects the fact that market dynamics is the aggregation of other market participants.} We assume players trade different stocks and the aggregation of strategies influences the representative player's spread through a spillover effect. $\rho^{(1)}$ and $\rho^{(2)}$ describe the resilience of the order book. In addition to submitting market orders in the traditional venue,  the representative player also submits passive orders into the dark pool, {where the execution cost is smaller than that in the traditional venue. However,}
the execution is uncertain. The execution times are described by a Poisson process $N$; {the occurance of jumps of $N$ corresponds to the occurance of order executions}. Thus, the current position $X^{(3)}$ follows

\vspace{-6.6mm}
\begin{equation}\label{position}
    X^{(3)}_t=\chi ^{(3)}+Z^{(1)}_t-Z^{(2)}_t+\int_0^tu_s\,dN_s,
\end{equation}

\vspace{-5.1mm}
where $\chi^{(3)}$ is the initial position of the representative player and $u$ is the net amount of passive orders submitted into the dark pool.  

\vspace{-1.6mm}
By using stategies $Z^{(1)}$ and $Z^{(2)}$, following \cite{HN-2014} the liquidity cost together with the cost crossing the spread is
$
	\int_0^T\left(X^{(1)}_{s-}+\frac{1}{2}\Delta(\kappa^{(1)}\nu^{(1)}_s+\eta^{(1)} Z^{(1)}_s)\right)\,dZ^{(1)}_s+	\int_0^T\left(X^{(2)}_{s-}+\frac{1}{2}\Delta(\kappa^{(2)}\nu^{(2)}_s+\eta^{(2)} Z^{(2)}_s)\right)\,dZ^{(2)}_s,
$
and the cost of spillover effect is assumed to be
$
	\frac{\kappa^{(1)}}{\eta^{(1)}}\int_0^T\left(X^{(1)}_{s-}+\frac{1}{2}\Delta(\kappa^{(1)}\nu^{(1)}_s+\eta^{(1)} Z^{(1)}_s)\right)\,d\nu^{(1)}_s+	\frac{\kappa^{(2)}}{\eta^{(2)}}\int_0^T\Big(X^{(2)}_{s-}+\frac{1}{2}\Delta(\kappa^{(2)}\nu^{(2)}_s+\eta^{(2)} Z^{(2)}_s)\Big)\,d\nu^{(2)}_s.
$
The ratio coefficients $\kappa^{(1)}/\eta^{(1)}$ and $\kappa^{(2)}/\eta^{(2)}$, coming from the dynamcis of the spreads, reflect the weight of influence between the aggregation and the individual strategy. The cost to minimize is given by

\vspace{-8.6mm}
\begin{equation}\label{cost-liquidation}
	\begin{split}
    &~\mathbb E\left[\int_0^T\left(X^{(1)}_{s-}+\frac{1}{2}\Delta(\kappa^{(1)}\nu^{(1)}_s+\eta^{(1)}{Z^{(1)}_s})\right)\,d(\kappa^{(1)}\nu^{(1)}_s+\eta^{(1)}{Z^{(1)}_s})\right.\\
    &~\left.+\int_0^T\left(X^{(2)}_{s-}+\frac{1}{2}\Delta(\kappa^{(2)}\nu^{(2)}_s+\eta^{(2)}{Z^{(2)}_s})\right)\,d(\kappa^{(2)}\nu^{(2)}_s+\eta^{(2)}{Z^{(2)}_s}){+\int_0^T(X^{(1)}_s-S^{(1)}_s)^2\,ds}\right.\\
    &~\left.{+\int_0^T(X^{(2)}_s-S^{(2)}_s)^2\,ds}+\int_0^T\lambda_s (X^{(3)}_s)^2\,ds+\int_0^T\gamma_su_s\,ds+\varrho (X^{(3)}_T)^2\right].
	\end{split}
\end{equation}

\vspace{-5.6mm}
The quadratic terms $\int_0^T\lambda_s (X^{(3)}_s)^2\,ds$ and $\varrho (X_T^{(3)})^2$ are inventory penalization, the term $\int_0^T\gamma_su_s\,ds$ is the cost arising from adverse selection, {the two terms $\int_0^T(X^{(i)}_s-S^{(i)}_s)^2\,ds$ $(i=1,2)$ are the penalization of deviation from the price signals $S^{(i)}$ $(i=1,2)$, which are assumed to be deterministic and c\`adl\`ag. The tracking of the price signals $S^{(i)}$ $(i=1,2)$ reflects the inverstor's anticipation of the market.	When tracking the c\`adl\`ag signals, intermediate and simultaneous jumps of $X^{(i)}$ (i=1,2) and thus of $Z^{(i)}$ (i=1,2) may happen.}
The goal is to find an equlibrium of the following MFG, which is a special case of \eqref{general-MFG}:
\begin{equation*}
	\left\{\begin{split}
		1.&~\textrm{Fix }(\nu^{(1)},\nu^{(2)})\textrm{ in some suitable space and minimize \eqref{cost-liquidation} subject to \eqref{buy-spread}-\eqref{position}};\\
		2.&~\textrm{Search for the fixed point }(\nu^{(1)},\nu^{(2)})=(\mathbb E[Z^{(1)}],\mathbb E[Z^{(2)}]),\\
		&~\textrm{ where } Z^{(1)}\textrm{ and }Z^{(2)}\textrm{ are the best response to }(\nu^{(1)},\nu^{(2)})\textrm{ in 1.}
	\end{split}\right.
\end{equation*}

\vspace{-2.6mm}
\textbf{\large{2.2\quad Optimal Exploitation of Exhaustible Resources}}

\vspace{-2.6mm}
In the second example, we consider an MFG of optimal exploitation of exhaustible resource. A model with infinite horizon and without game nature was introduced in \cite{Ferrari2018}, {where Ferrari and Koch studied a single player's optimal extraction problem by solving a two-dimensional degenerate singular control problem with finite fuel via a combination of calculaus of variation established in \cite{BBF-2018} and the standard approach}. We introduce a model among infinite players with mean-field interaction and finite horizon. {In the model, each player is endowed with limited amount of exhaustible resource, such as earth minerals, metal ores and fossil fuels, to exploit for sale.} 

{Let $X_t$ be the reservoir of the resource at time $t$, and $Z_t$ be the accumulative amount of exploitation until time $t$. Thus,}
\begin{equation}\label{reservoir}
	X_t=x-Z_t,
\end{equation}
{where $x$ is the initial reservior.}
{One character of problems with exhaustible resource is that $X_t$ cannot be negative due to its ecnomic meaning. One way to address this issue is to add an absorption boundary at $0$: the player drops out of the game once $X_t=0$; see \cite{PUR-2021}. Another way to address this issue is to incorporate singular controls with finite fuels, i.e., $Z_t$ is assumed to be valued in $[0,x]$ for each $t$. } \\[-2em] 

The market price of the resource is determined by three parts: the first part is generated from the market itself and noise traders. It is assumed to be a mean-reverting process. When there is no exploitation activity, the price would recover to the mean level. The second part comes from the player's sales. Once selling, the price is moved in an undesirable direction due to illiquidity of the exhaustible resource. The third part arises from alternative resource. Any exploitation of alternative resource would make the price of the exhaustible resource decline. We assume the price impact to be in a linear form. Hence, the actual market price of the exhaustible resource follows 
	$dP_t=(a-bP_t)\,dt+\sigma\,dW_t-\eta\,dZ_t-\kappa\,d\nu_t,$
where $\nu$ is the aggregated (mean field) exploitation of alternative resource. The goal for the representative player is to maximize the profit

\vspace{-6.6mm}
\begin{equation}\label{profit}
	\mathbb E\left[\int_0^TP_t\,d(\eta Z_t+\kappa\nu_t)^c+\sum_{0\leq t\leq T}\int_0^{ \Delta(\eta Z_t+\kappa\nu_t) }(P_{t-}-x)\,dx\right],
\end{equation}
where $(\eta Z+\kappa \nu)^c$ is the continuous part of $\eta Z+\kappa\nu$. To maximize \eqref{profit} is equivalent to minimize
\begin{equation}\label{cost-resource}
	\mathbb E\left[\int_0^T\overline P_t\,d(\eta Z_t+\kappa\mu_t)^c+\sum_{0\leq t\leq T}\int_0^{\Delta(\eta Z_t+\kappa\nu_t) }(\overline P_{t-}+x)\,dx\right],
\end{equation}
where
\begin{equation}\label{minus-price}
	d\overline P_t=(-a-b\overline P_t)\,dt-\sigma\,dW_t+\eta\,dZ_t+\kappa\,d\nu_t.
\end{equation}
Therefore, the following MFG is a special case of \eqref{general-MFG}:
\begin{equation*}
	\left\{\begin{split}
		1.&~\textrm{Fix }\nu\textrm{ in some suitable space and minimize }\eqref{cost-resource}\textrm{ subject to }\eqref{reservoir}\textrm{ and }\eqref{minus-price};\\
		2.&~\textrm{Search for the fixed point }\nu=\mathbb E[Z],\textrm{ where }Z\textrm{ is the optimal control from 1}.
	\end{split}\right.
\end{equation*}
\vspace{-7.6mm}

{The analysis in our forthcoming paper \cite{FHX2021} yields a characterization of the equilibrium for a modified version of the MFG in Section 2.1 without passive orders but with random volatility $\sigma$ and liquidation constraint. In this paper, we are motivated by MFGs in Section 2 to study a general class of MFGs \eqref{general-MFG}.}

\vspace{-6.6mm}
\section{Extended MFGs with Singular Controls}\label{sec:model-setup}

\vspace{-5.6mm}
\textbf{Space and Filtration.} Throughout the paper, denote by $\mathcal D([0,T];\mathbb R^d)$ the Skorokhod space of all functions from $[0,T]$ to $\mathbb R^d$ with c\`adl\`ag path, by $\mathcal C([0,T];\mathbb R^d)\subset\mathcal D([0,T];\mathbb R^d)$ the subset of all continuous functions and by $\mathcal A^m([0,T];\mathbb R^d)\subset \mathcal D([0,T];\mathbb R^d)$ the subset of all non-decreasing functions with $z_T\leq m$ and $m\in(0,\infty]$, which is understood in the componentwise sense $z^j_T\leq m$, $j=1,\cdots,d$. To incorporate the initial and final jumps of elements in $\mathcal D([0,T];\mathbb R^d)$, we identify trajectories on $[0,T]$ with ones on the whole real line by the following trivially extended space

\vspace{-7.1mm}
\[
	\widetilde{\mathcal D}_{0,T}(\mathbb R;\mathbb R^d):=\{x\in {\mathcal D}(\mathbb R;\mathbb R^d): x_t=0\textrm{ for }t<0\textrm{ and }x_t=x_T\textrm{ for }t>T\}.
\]

\vspace{-3.9mm}
Correspondingly, we can define $\widetilde{\mathcal C}_{0,T}(\mathbb R;\mathbb R^d)$ and $\widetilde{\mathcal A}^m_{0,T}(\mathbb R;\mathbb R^d)$. For any metric space $(S,\varrho)$, denote by $\mathcal M_+(S;\varrho)$ the set of all finite non-negative measures on $S$ and by $\mathcal P(S;\varrho)\subset\mathcal M_+(S;\varrho)$ the set of all probability measures on $S$ and by $\mathcal P_p(S;\varrho)$ the subset of probability measures with finite $p$-th moments. When the metric $\varrho$ is clear from the context, we write $\mathcal M_+(S)$, $\mathcal P(S)$ and $\mathcal P_p(S)$ for short. Denote by $\mathcal U([0,T]\times U)\subset\mathcal M_+([0,T]\times U)$ the set of all measures on $[0,T]\times U$ with the first marginal Lebesgue measure on $[0,T]$ and the second marginal a probability measure on $U$, where $U$ is some metric space. Similarly, we identify $\mathcal U([0,T]\times U)$ with $\widetilde{\mathcal U}_{0,T}(\mathbb R\times U)$, where

\vspace{-10.1mm}
\[
	\widetilde{\mathcal U}_{0,T}(\mathbb R\times U)=\{q\in\mathcal U(\mathbb R\times U):1_{(-\infty,0)\times U}q(dt,du)=\delta_{u_0}(du)dt\textrm{ and }1_{(T,\infty)\times U}q(dt,du)=\delta_{u_T}(du)dt\}
\]

\vspace{-5.1mm}
for some fixed $u_0,u_T\in U$. Each element $q\in\widetilde{\mathcal U}_{0,T}(\mathbb R\times U)$ admits the disintegration $q(dt,du)=q_t(du)dt$. 
When there is no confusion, we write $\widetilde{\mathcal D}_{0,T}$, $\widetilde{\mathcal C}_{0,T}$, $\widetilde{\mathcal A}^m_{0,T}$ and $\widetilde{\mathcal U}_{0,T}$ for simplicity. Let the canonical space be defined as the product space
$	\Omega^m:=\widetilde{\mathcal D}_{0,T}\times \widetilde{\mathcal D}_{0,T}\times \widetilde{\mathcal D}_{0,T}\times\widetilde{\mathcal U}_{0,T}\times\widetilde{\mathcal A}^m_{0,T}\times \widetilde{\mathcal A}^m_{0,T}$,
and let $(X^{(1)},X^{(2)},X^{(3)},Q,Z^{(1)},Z^{(2)})$ be the coordinate processes on $\Omega^m$, i.e.,

\vspace{-6.6mm}
\[
	X^{(1)}(\omega)=x^{(1)}, ~X^{(2)}(\omega)=x^{(2)},~ X^{(3)}(\omega)=x^{(3)},~ Q(\omega)=q, ~Z^{(1)}(\omega)=z^{(1)},~ Z^{(2)}(\omega)=z^{(2)},
\]
for each $\omega=(x^{(1)},x^{(2)},x^{(3)},q,z^{(1)},z^{(2)})\in\Omega^m$. Note that by \cite[Section 2.1.2]{FH-2017} and  \cite[Lemma 3.2]{L-2015}, $Q$ can be identified by a predictable disintegration in the following sense
$
Q(dt,du)=Q_t(du)dt.
$
The space $\Omega^m$ is equipped with the product $\sigma$-algebra $\mathcal F_t=\mathcal F^{X^{(1)}}_t\times \mathcal F^{X^{(2)}}_t\times \mathcal F^{X^{(3)}}_t\times\mathcal F_t^Q\times\mathcal F_t^{Z^{(1)}}\times\mathcal F_t^{Z^{(2)}}$, where $\mathcal F^{X^{(i)}}_t$ is the $\sigma$-algebra generated by the $\Pi$ system $\{\{x\in\widetilde{\mathcal D}_{0,T}: (x_{t_1},\cdots,x_{t_n})\in A_1\times \cdots\times A_n\}|~  t_1\leq\cdots\leq t_n\leq t,A_j\in\mathcal B(\mathbb R^{d}),n\in\mathbb N\}$, $\mathcal F^{Z^{(i)}}_t$ is the $\sigma$-algebra generated by the $\Pi$ system $\{\{z\in\widetilde{\mathcal A}^m_{0,T}: (z_{t_1},\cdots,z_{t_n})\in A_1\times\cdots\times A_n\}|~ t_1\leq\cdots\leq t_n\leq t,A_j\in\mathcal B(\mathbb R^{d}),n\in\mathbb N\}$ and $\mathcal F^Q_t$ is the $\sigma$-algebra generated by $1_{[0,t]}\underline Q$, where $\underline Q$ is the coordinate projection from $\widetilde{\mathcal U}_{0,T}$ to itself, i.e., $\underline Q(q)=q$ for each $q\in\widetilde{\mathcal U}_{0,T}$.

\vspace{-1.6mm}
\textbf{Metric. }Let
$|y|$, $\|x\|:=\max_{1\leq j\leq d}|x^j|$ and $|u|_U$ be the norm of $y\in\mathbb R$, $x\in\mathbb R^d$ and $u\in U$, respectively. For $x,y\in\mathbb R^d$, denote by $x\cdot y$ the inner product of $x$ and $y$. For each $t$, $\|x\|_t=\sup_{0\leq s\leq t}\|x_s\|$ denotes the uniform norm of $x\in\mathcal C([0,t];\mathbb R^d)$ (and thus of $x\in\widetilde{\mathcal C}_{0,t}$). Endow $\widetilde{\mathcal D}_{0,T}$ and $\widetilde{\mathcal A}^m_{0,T}$ with Skorokhod weak $M_1$ topology and endow $\mathcal P_p(S;\varrho)$ with Wasserstein metric $\mathcal W_{p,(S,\varrho)}$.  By Proposition \ref{prop:appendix} in Appendix \ref{app:complete} the spaces $\widetilde{\mathcal{D}}_{0,T}$ and $\widetilde{\mathcal{A}}^m_{0,T}$ are Polish when endowed with the $M_1$ topology.
%
%
It is well-known that $(\mathcal P_p(S),\mathcal W_{p,(S,\varrho)})$ is Polish if $(S,\varrho)$ is Polish. Endowed with the following metric induced by Wasserstein metric $\widetilde{\mathcal U}_{0,T}$ is Polish:

\vspace{-10.6mm}
\[
	\widetilde{\mathcal W}_{p,[0,T]\times U}(q_1,q_2)=\mathcal W_{p,[0,T]\times U}\left(\frac{q_1}{T},\frac{q_2}{T}\right)+\sum_{n=0}^\infty\frac{1}{2^{n+1}}\left( \mathcal W_{p,[T+n,T+n+1]\times U}(q_1,q_2)+\mathcal W_{p,[-(n+1),-n]\times U}(q_1,q_2)\right).
\]

\vspace{-4.3mm}
\textbf{Convention. }We use the convention that $C$ is a generic constant which may vary from line to line. For a stochastic process $X$ by $X\in\widetilde{\mathcal D}_{0,T}$ we mean $X(\omega)\in\widetilde{\mathcal D}_{0,T}$ a.s.; other analogous notation can be understood in the same way. Whenever we mention $W$, $\mu$, $Z$ and $X$, we mean $(W^{(1)},W^{(2)})$, $(\mu^{(1)},\mu^{(2)},\mu^{(3)},\mu^{(4)},\mu^{(5)})$, $(Z^{(1)},Z^{(2)})$ and $(X^{(1)},X^{(2)},X^{(3)})$, respectively,  unless otherwise stated; the same convention holds for other variants of $(W,~\mu,~Z,~X)$ like $(\widetilde W,~\widetilde\mu,~\widetilde Z,~\widetilde X)$, $(W^k,~\mu^k,~Z^k,~X^k)$ etc.
Moreover, for each $\nu\in\mathcal{P}_p(\widetilde{\mathcal{D}}_{0,T})$, put $\nu_t=\nu\circ\pi^{-1}_t$, where $\pi_t:x\in\widetilde{\mathcal{D}}_{0,T}\rightarrow x_t$ and $\overline{\nu}:=\int x\nu(dx)$. 

We are ready to introduce the notion of relaxed controls.
\begin{definition}\label{control-rule-mu}
A probability measure $\mathbb{P}$ on $\Omega^m$ is called a relaxed control with respect to $\mu\in\mathcal{P}_p(\widetilde{\mathcal{A}}^m_{0,T})\times \mathcal{P}_p(\widetilde{\mathcal{A}}^m_{0,T})\times \mathcal{P}_p(\widetilde{\mathcal{D}}_{0,T})\times \mathcal{P}_p(\widetilde{\mathcal{D}}_{0,T})\times \mathcal{P}_p(\widetilde{\mathcal{D}}_{0,T})$ if

\vspace{-2.6mm}
\quad 1. $(X,Q,Z)$ are coordinate processes on the canonical space $\Omega^m$;

\vspace{-2.6mm}
\quad 2. there exists an adapted process $Y\in\widetilde{\mathcal D}_{0,T}$ such that\\[-1.1em]
        \begin{equation*}
            \begin{split}
                1)&~\mathbb P\left(Y^{}=X^{(3)}-\alpha^{(1)} Z^{(1)}+\alpha^{(2)}Z^{(2)}\right)=1,\\
                2)&~\mathcal M^{\phi,X^{(1)},Z^{(1)},\mu^{(1)}} \textrm{ and }\mathcal M^{\phi,X^{(2)},Z^{(2)},\mu^{(2)}}\textrm{ are continuous }\mathbb P \textrm{ martingales}, \textrm{ for each }\phi\in\mathcal C^2_b(\mathbb R^d;\mathbb R),\\
                3)&~\mathcal M^{\phi,X^{(3)},Y^{},Q} \textrm{ is a  }\mathbb P\textrm{ martingale with c\`adl\`ag path},\textrm{ for each } \phi\in \mathcal{C}^2_b(\mathbb{R}^d;\mathbb R),
            \end{split}
        \end{equation*}
        where \\[-7.6mm] 
     
  \qquad    $\bullet$ $\mathcal{C}^2_b(\mathbb{R}^d;\mathbb R)$ is the space of all continuous and bounded functions from $\mathbb R^d$ to $\mathbb R$ with continuous and bounded first- and second-order derivatives,\\[-7.6mm] 
        
    \qquad    $\bullet$ for $t\in[0,T]$ and $i=1,2$ \\[-3.6mm]
         \begin{equation*}
        	\begin{split}
        	\mathcal M^{\phi,X^{(i)},Z^{(i)},\mu^{(i)}}_t:=&~\phi(X^{(i)}_t)-\int_0^t\mathbb L^{(i)}\phi(s,X_s^{(i)})\,ds-\int_0^t\nabla\phi(X^{(i)}_s)\cdot\,d(\kappa^{(i)}\overline\mu^{(i)}_s+\eta^{(i)}Z^{(i)}_s)\\
        	&~-\sum_{0\leq s\leq t}\left(\phi(X_s)-\phi(X_{s-})-\nabla\phi(X_{s-})\cdot\Delta X_s\right),
        	\end{split}
        	\end{equation*}
        	with $\mathbb L^{(i)}\phi(s,x)=b^{(i)}(s,x,\mu^{(i)})\cdot \nabla\phi(x)+\frac{1}{2}Tr(a^{(i)}(s)\Delta\phi(x))$, $a^{(i)}=\sigma^{(i)}(\sigma^{(i)})^\top$, \\[-6.6mm]
        
        \qquad $\bullet$ and for $t\in[0,T]$,
          $\mathcal{M}^{\phi,X^{(3)},Y^{},Q}_t
            :=\phi(Y^{}_t)-\int_0^t\int_U\mathcal L^{}\phi(s,X^{(3)}_{s-},Y^{}_{s-},u)\,Q_s(du)ds,$
             with $\mathcal L^{}\phi(s,x,y,u):=\nabla\phi(y)\cdot b^{(3)}(s,x,u)+\big(\phi(y+l^{}(s,u))-\phi(y)-\nabla\phi(y)\cdot l^{}(s,u)\big)\lambda^{}_s$.
\end{definition}
 \begin{remark}
 \begin{itemize}
 \item[(1)] In \cite{FH-2017}, the probability measure $\mathbb P$ defined on the canonical space is called a \textit{control rule}. Here we do not distinguish control rule and relaxed control since there is no confusion.\\[-1.9em] 
 \item[(2)] The definition of relaxed controls (control rules) is different from \cite{FH-2017}. The current definition can avoid considering the simultaneous jumps of singular controls and the Poisson process in the definition of $\mathcal M^{\phi,X^{(3)},Y,Q}$. Definition \ref{control-rule-mu} is linked to the weak solution of SDEs given by the following proposition. The proof is the same as \cite[Lemma 2.1]{BCP-2017}.
 \end{itemize}
 \end{remark}
 \begin{proposition}\label{martingale-representation}
The probability measure $\mathbb P$ on $\Omega^m$ is a relaxed control if and only if there is an extension of $(\Omega,\mathcal F,\{\mathcal F_t\},\mathbb P)$, $(\widehat\Omega,\widehat{\mathcal F},\{\widehat{\mathcal F}_t\},\widehat{\mathbb P})$, on which a tuple of adapted stochastic processes $(\widehat X,\widehat Z,\widehat Q,\widehat W,\widehat N)$ is defined such that {for all $t\in[0,T]$\footnote{Note that the time horizon $[0,T]$ only depends on the horizon of the corresponding martingale problem. In Section \ref{proof} the horizon of the martingale problem is enlarged to $[0,T+1]$, so is the horizon of the weak solution of SDEs.}}
\begin{equation}\label{resume-SDE-12}
    d\widehat X^{(i)}_t=b^{(i)}(t,\widehat X^{(i)}_t,\mu^{(i)}_t)\,dt+\,d\left(\eta^{(i)}\widehat Z^{(i)}_t+\kappa^{(i)}\overline\mu^{(i)}_t\right)+\sigma^{(i)}(t)\,d\widehat W^{(i)}_t,~i=1,2, \textrm{ and}
\end{equation}
\begin{equation}\label{resume-SDE-3}
	\begin{split}
    d\widehat X^{(3)}_t=&~\int_Ub^{(3)}(t,\widehat X^{(3)}_t,u)\widehat Q(dt,du)+\alpha^{(1)}\,d\widehat Z^{(1)}_t-\alpha^{(2)}\,d\widehat Z^{(2)}_t+\int_Ul^{}(t,u)\widetilde {\widehat N^{}}(dt,du),
	\end{split}
\end{equation}
where $\widehat W^{(1)}$ and $\widehat W^{(2)}$ are two Brownian motions, and $\widetilde {\widehat N^{}}$ is a compensated Poisson random measure with intensity $\lambda^{}_t\widehat Q(dt,du)$. Moreover, two tuples are related by $\mathbb P\circ(X,Q,Z)^{-1}=\widehat{\mathbb P}\circ (\widehat X,\widehat Q,\widehat Z)^{-1}$.
 \end{proposition}

\vspace{-2.6mm}
Given $\mu\in\mathcal{P}_p(\widetilde{\mathcal{A}}^m_{0,T})\times\mathcal{P}_p(\widetilde{\mathcal{A}}^m_{0,T})\times\mathcal{P}_p(\widetilde{\mathcal{D}}_{0,T})\times\mathcal{P}_p(\widetilde{\mathcal{D}}_{0,T})\times\mathcal{P}_p(\widetilde{\mathcal{D}}_{0,T})$, the set of relaxed controls associated with $\mu$ is denoted by $\mathcal{R}^m(\mu)$, and the cost associated with a relaxed control $\mathbb{P}\in\mathcal{R}^m(\mu)$ is given by

\vspace{-8.6mm}
    \begin{equation}\label{cost-relaxed-control}
            \begin{split}
        J(\mathbb P;\mu)=&~\mathbb E^{\mathbb P}\left[\sum_{i=1}^2\int_0^Th(X^{(i)}_s)\cdot\,d(\kappa^{(i)}\overline\mu^{(i)}_s+\eta^{(i)}Z_s^{(i)})^c+\sum_{i=1}^2\sum_{0\leq t\leq T}\sum_{j=1}^d\int_0^{\Delta X^{(i)}_{j,t}}h_j(X^{(i)}_{j,t-}+x)\,dx\right.\\
        &~\left.+\int_0^T\int_Uf(t,X_t,\mu_t,u_t)\,Q_t(du)dt+g(X_T,\mu_T)\right].
        \end{split}
        \end{equation}
    
    \vspace{-5.1mm}
The set of optimal relaxed controls associated with $\mu$ is denoted by
$
	\mathcal{R}^{m,*}(\mu):=\textrm{argmin}_{\mathbb P\in\mathcal{R}^m(\mu)}J(\mathbb{P};\mu).
$
Based on the notion of relaxed controls, we introduce the definition of relaxed solutions to MFGs. If a probability measure $\mathbb{P}$ satisfies the fixed point property
	$\mathbb{P}\in \mathcal{R}^{m,*}\Big(\mathbb{P}\circ (Z^{(1)})^{-1},\mathbb{P}\circ (Z^{(2)})^{-1},\mathbb P\circ(X^{(1)})^{-1},\mathbb P\circ(X^{(2)})^{-1},\mathbb P\circ(X^{(3)})^{-1}\Big),$
then we call $\mathbb{P}$ or the associated tuple $(\Omega^m,\mathcal{F},\{\mathcal{F}_t\},\mathbb{P},X,Q,Z)$ a relaxed solution to the MFG with singular controls \eqref{general-MFG}. Moreover, if $\mathbb{P}\in \mathcal{R}^{m,*}\Big(\mathbb{P}\circ (Z^{(1)})^{-1},\mathbb{P}\circ (Z^{(2)})^{-1},\mathbb P\circ(X^{(1)})^{-1},\mathbb P\circ(X^{(2)})^{-1},\mathbb P\circ(X^{(3)})^{-1}\Big)$ and $\mathbb{P}(Q(dt,du)=\delta_{\tilde{u}_t}(du)dt)=1$ for some progressively measurable process $\tilde{u}$, then we call $\mathbb{P}$ or the associated tuple $(\Omega^m,\mathcal{F},\{\mathcal{F}_t\},\mathbb{P},X,\bar{u},Z)$ a strict solution.

To guarantee the existence of a relaxed solution to \eqref{general-MFG}, we make the following assumptions.

\vspace{-5.6mm}
 \begin{itemize}
	\item[$\mathcal{A}_1$. ] The $\mathbb R^d$ valued functions $b^{(1)},~b^{(2)}$ and $b^{(3)}$ are measurable in $t\in[0,T]$ and there exists a positive constant $C_1$ such that $\|b^{(1)}(t,x,\nu)\|+\|b^{(2)}(t,x,\nu)\|\leq C_1(1+\|x\|+\mathcal W_p(\nu,\delta_0))$, $\|b^{(3)}(t,x,u)\|\leq C_1(1+\|x\|)$ and
	$\|b^{(1)}(t,x,\nu)-b^{(1)}(t,y,\nu)\|+\|b^{(2)}(t,x,\nu)-b^{(2)}(t,y,\nu)\|+\|b^{(3)}(t,x,u)-b^{(3)}(t,y,u)\|\leq C_1\|x-y\|$, for any $(t,x,y,u,\nu)\in[0,T]\times\mathbb R^d\times\mathbb R^d\times U\times\mathcal P_p(\mathbb R^d)$. Morover, $b^{(3)}$ is continuous in $u$.
	\item[$\mathcal{A}_2$. ] The function $f$ is measurable in $t\in[0,T]$ and continuous with respect to $(x,\nu,u)\in(\mathbb{R}^d)^3\times(\mathcal{P}_p(\mathbb{R}^d))^5\times U$. $g$ is continuous in $(x,\nu)\in(\mathbb{R}^d)^3\times(\mathcal P_p(\mathbb R^d))^5$. $h(y):=(h_1(y_1),\cdots,h_d(y_d))$ for each $y\in\mathbb R^d$ and each $h_i\in\mathcal C^1(\mathbb R)$, the space of continuous functions on $\mathbb R $ with continuous derivatives.
	\item[$\mathcal{A}_3$. ]  For $p\geq 1$, there exists a positive constant $C_2$  such that for each $(t,x,y,\nu,u)\in [0,T]\times(\mathbb{R}^d)^3\times\mathbb R\times(\mathcal{P}_p(\mathbb{R}^d))^5\times U$
	\[
	|h_i(y)|+ |h'_i(y)|\leq C_2\left(1+|y|^{p-1}\right),\quad i=1,\cdots,d,
	\]
	\[
	|g(x,\nu)|\leq C_2\left(1+\|x\|^{{p}}+\mathcal W_p^p(\nu,\delta_0)\right)\quad
	\textrm{ and }\quad
	|f(t,x,\nu,u)|\leq C_2\left(1+\|x\|^{p}+\mathcal W_p^{ p}(\nu,\delta_0)\right),
	\]
	where $\mathcal W_p(\nu,\delta_0):=\left(\sum_{i=1}^5\mathcal W_p^{ p}(\nu^{(i)},\delta_0)\right)^{\frac{1}{p}}$.
	\item[$\mathcal{A}_4$. ] $(\alpha^{(1)},\alpha^{(2)},\kappa^{(1)},\kappa^{(2)},\eta^{(1)},\eta^{(2)})\in\mathbb R^6$. $(\sigma^{(1)}, \sigma^{(2)}):[0,T]\rightarrow\mathbb R^{d}\times\mathbb R^{d}$ are bounded and measurable. Denote $a^{(1)}=\sigma^{(1)}(\sigma^{(1)})^\top$ and $a^{(2)}=\sigma^{(2)}(\sigma^{(2)})^\top$. $\lambda^{}:[0,T]\rightarrow(0,\infty)$ are measurable and bounded. $l^{}$ is a bounded and measurable function on $[0,T]\times U$ and continuous in $u$. $\kappa^{(i)}\eta^{(i)}\geq 0$, $i=1,2$.
	\item[$\mathcal{A}_5$. ] The functions $b^{(1)}$, $b^{(2)}$ and $f$ are locally Lipschitz continuous with measures uniformly in other arguments i.e., there exists $C_3>0$ such that for each $(t,x,y,u)\in[0,T]\times(\mathbb{R}^d)^3\times\mathbb{R}^d\times U$, $\nu^1,\nu^2\in(\mathcal P_p(\mathbb R^d))^5$ and $\nu^{1'},\nu^{2'}\in\mathcal P_p(\mathbb R)$ there holds that
	
	\vspace{-7.6mm}
	\begin{equation}
		\begin{split}
	&~|f(t,x,\nu^1,u)-f(t,x,\nu^2,u)|\leq C_3\Big(1+L(\mathcal{W}_p(\nu^1,\delta_0),\mathcal{W}_p(\nu^2,\delta_0))\Big) \mathcal{W}_p(\nu^1,\nu^2),\\
	&~|b^{(i)}(t,y,\nu^{1'})-  b^{(i)}(t,y,\nu^{2'})	| \leq C_3\Big(1+L(\mathcal{W}_p(\nu^{1'},\delta_0),\mathcal{W}_p(\nu^{2'},\delta_0))\Big) \mathcal{W}_p(\nu^{1'},\nu^{2'}), 
	\end{split}	
\end{equation}
	where $L(\mathcal{W}_p(*,\delta_0),\mathcal{W}_p(**,\delta_0))$ is locally bounded with $\mathcal{W}_p(*,\delta_0)$ and $\mathcal{W}_p(**,\delta_0)$. 
	\item[$\mathcal{A}_6$. ] $U$ is a compact metrizable space. 
\end{itemize}

\vspace{-3.6mm}
The following two theorems are our two main results. The proofs of them are given in Section \ref{proof} and Section \ref{sec:general-control}, respectively.

\begin{theorem}[Existence with finite fuel constraint]\label{existence}\label{thm:existence-finite-fuel}
Under assumptions $\mathcal A_1$-$\mathcal A_6$, there exists a relaxed solution to MFGs with singular controls \eqref{general-MFG} for each $0<m<\infty$.
\end{theorem}
\begin{theorem}[Existence with general singular controls]\label{thm:general-control}
	In addition to assumptions $\mathcal A_1$-$\mathcal A_6$, we assume the following assumption $\mathcal A_7$ holds.
	\newline{$\mathcal A_7$.} For $i=1,2$, $\eta^{(i)}=0$ if and only if $\alpha^{(i)}=0$. For $p\geq 2$, there exists a positive constant $C_4$ such that the following coercive conditions hold for $i=1,2$, $j=1,\cdots,d$, $\hat x\in\mathbb R$, $x\in\mathbb R^d$, $\tilde x\in(\mathbb R^{d})^3 $, $\nu\in\mathcal P_p(\mathbb R^d)$, $\tilde\nu\in(\mathcal P_p(\mathbb R^{d}))^5$ and $u\in U$
	
	\vspace{-5.6mm}
	\begin{equation}\label{ass:coercive}
		\left\{\begin{split}
		&~-C_4(1-|\hat x|^{p})\leq \int_0^{\hat x} h_j(r)\,dr\leq C_4(1+|\hat x|^{p}),\\
		&~	-C_4(1-|x_j|-\mathcal W_p(\nu,\delta_0))	\leq b_j^{(i)}(t,x,\nu)\leq C_4(	1+|x_j|+\mathcal W_p(\nu,\delta_0)	), \\
	&~-C_4(1-\|x\|^{p}-\mathcal W^p_p(\nu,\delta_0))\leq-h(x)\cdot b^{(i)}(t,x,\nu)\leq C_4(1+\|x\|^{p}+\mathcal W^p_p(\nu,\delta_0)),  \\
	&~-C_4(1-\|\tilde x\|^{p}-\mathcal W_p^p(\tilde\nu,\delta_0))\leq g(\tilde x,\tilde\nu)\leq C_4(1+\|\tilde x\|^{p}+\mathcal W_p^p(\tilde\nu,\delta_0)),\\
	&~-C_4(1-\|\tilde x\|^{p}-\mathcal W_p^{ p}(\tilde\nu,\delta_0))\leq f(t,\tilde x,\tilde\nu,u)\leq C_4\left(1+\|\tilde x\|^{ p}+\mathcal W_p^p(\tilde\nu,\delta_0)\right).
		\end{split}\right.
	\end{equation} 
{Moreover, for some $ 1<\bar p<p$, the continuity of the coefficients with the measure is in $\mathcal W_{\bar p}$:		
	\begin{equation}\label{ass:convergence-lower-order}
	\left\{\begin{split}
		&~g(x^n,\nu^n)\rightarrow g(x,\nu)\textrm{ if }x^n\rightarrow x\textrm{ and }\nu^n\rightarrow\nu\textrm{ in }\mathcal W_{\bar p},\\
		&~|f(t,x,\nu^1,u)-f(t,x,\nu^2,u)|\leq C_3\Big(1+L(\mathcal{W}_p(\nu^1,\delta_0),\mathcal{W}_p(\nu^2,\delta_0))\Big) \mathcal{W}_{\bar p}(\nu^1,\nu^2),\\
		&~|b^{(i)}(t,y,\nu^{1'})-  b^{(i)}(t,y,\nu^{2'})	| \leq C_3\Big(1+L(\mathcal{W}_p(\nu^{1'},\delta_0),\mathcal{W}_p(\nu^{2'},\delta_0))\Big) \mathcal{W}_{\bar p}(\nu^{1'},\nu^{2'}), 
	\end{split}	\right.
\end{equation}					}
Then there exists a relaxed solution to \eqref{general-MFG} when $m=\infty$.
\end{theorem}

\begin{remark}\label{existence-strict-MFG} 
	(1). Examples in Section \ref{sec:motivation} satisfy assumptions $\mathcal A_1$-$\mathcal A_7$.
\newline (2). By \cite[Remark 2.8]{FH-2017}, additional convexity assumption implies that a strict solution to MFGs can be constructed from a relaxed solution.
\newline (3).
In Theorem \ref{thm:existence-finite-fuel} and Theorem \ref{thm:general-control} we consider all processes starting from $0$ for simplicity. The extension to general and different initial values is straightforward. However, there is one generalization our analysis cannot address; neither $\alpha^{(i)}$ nor $\eta^{(i)}$ is allowed to depend on the state, control or measure, unless our model \eqref{general-MFG} is modified. The reason is the possible simultaneous jumps of the intergrant and the differential function. For example, if $Z^n\rightarrow Z$ in $M_1$, generally it is not true that $\int_0^\cdot Z^n_s\,dZ^n_s\rightarrow\int_0^\cdot Z_s\,dZ_s$. 
\newline {(4). 
	It can be proved that the relaxed solution obtained in Theorem \ref{thm:general-control} can be used to construct an approximate equilibrium of an $N$-player game. {It is worth noting that we do not distinguish regular controls and singular controls in discrete time models. For this reason, MFGs in discrete time can be viewed as MFGs with singular controls. Regarding the problem of approximate equilibria of $N$-player games in discrete time, we may refer to e.g. \cite{GX-2019,Saldi2018}, among many others.} However, it is difficult to prove the reverse convergence; refer to \cite{Fischer2017,Lacker-2016} for the weak convergence without singular controls. For the convergence from $N$-player games to MFGs with singular controls, the only result to our knowledge is \cite{GX-2018}, where the (strong) convergence of value functions was obtained by the explicit solution of the $N$-player game. For the general characterization of the weak convergence, it is open in the literature. 
	We plan to study it in an independent work, together with the (well-established) approximate relaxed Nash equilibrium result.} 
\end{remark}

\vspace{-7.6mm}
\section{Existence of Equilibria with Finite Fuel Constraint}\label{proof}

\vspace{-5.6mm}
In this whole section (Section \ref{section-finite-fuel-smooth} and \ref{sec:approximation}), we prove the existence of a relaxed solution to MFGs under a finite fuel constraint, i.e., Theorem \ref{thm:existence-finite-fuel}. That is, in this section the space of admissible singular controls is
$\widetilde{\mathcal{A}}^m_{0,T}:=\{z\in\widetilde{\mathcal{A}}_{0,T}:\|z_T\|\leq m\},$ 
for some $m\in(0,\infty)$. By \cite[Theorem 12.12.2]{Whitt-2002}, the set $\widetilde{\mathcal{A}}^m_{0,T}$ is $(\widetilde{\mathcal{D}}_{0,T},{M_1})$ compact.

\vspace{-1.6mm}
As mentioned in the introduction, due to the possible simultaneous jumps of $Z^{(1)}$, $Z^{(2)}$ and the Poisson process, it is difficult to show the tightness of $X^{(3)}$. 
We circumvent the problem by spliting the proof of Theorem \ref{existence} into two parts. In Section \ref{section-finite-fuel-smooth} we prove the existence of equilibria by smoothing the singular controls $Z^{(1)}$ and $Z^{(2)}$. Thus, the tightness of $X^{(3)}$ can be obtained in Section \ref{section-finite-fuel-smooth}. The general case is considered in Section \ref{sec:approximation} using an approximation argument. Note that the tightness of $X^{(3)}$ is necessary in Section \ref{section-finite-fuel-smooth} while we do not need it in Section \ref{sec:approximation}.

Precisely, in Section \ref{section-finite-fuel-smooth}, instead of singular control $Z$ we consider its continuous counterpart $Z^{[k]}_{t}:=k\int_{t-1/k}^t Z_s\,ds$. 
Since $T+1$ is definitely a continuous point of $Z^{(i)}$, by \cite[Theorem 12.9.3(ii)]{Whitt-2002}, it holds that
	$	Z^{(i),[k]}\rightarrow Z^{(i)}\textrm{ in }(\widetilde{\mathcal A}_{0,T+1}, M_1),~i=1,2$,
which is not necessarily true in $(\widetilde{\mathcal A}_{0,T}^m,M_1)$ since $T$ might be a discontinuous time point of $Z^{(i)}$.  Therefore, in Section \ref{section-finite-fuel-smooth} the canonical space is chosen as  

\vspace{-6.6mm}
$$\Omega^{m,o}:=\widetilde{\mathcal D}_{0,T+1}\times\widetilde{\mathcal D}_{0,T+1}\times\widetilde{\mathcal D}_{0,T+1}\times\widetilde{\mathcal U}_{0,T}\times\widetilde{\mathcal A}_{0,T}^m\times\widetilde{\mathcal A}_{0,T}^m.$$ Let $X^o$ be the coordinate projection onto $\widetilde{\mathcal D}_{0,T+1}\times\widetilde{\mathcal D}_{0,T+1}\times\widetilde{\mathcal D}_{0,T+1}$ and $(\mathcal F^o_t)$ be the canonical filtration on $\Omega^o$. Correspondingly, we extend the time-domain of coefficients from $[0,T]$ to $[0,T+1]$, i.e., let $\widetilde\gamma$ satisfy the same assumptions as $\mathcal A_1$-$\mathcal A_6$ with $[0,T]$ replaced by $[0,T+1]$ such that
\begin{equation}\label{extension}
\widetilde \gamma(t,\cdot)=
\gamma(t,\cdot),\quad\textrm{on }t\in[0,T],
\end{equation}
where $\widetilde\gamma=\widetilde b^{(1)},~\widetilde b^{(2)},~\widetilde b^{(3)},~\widetilde\sigma^{(1)},~\widetilde\sigma^{(2)},~\widetilde\lambda^{},~\widetilde l^{}$ and $\gamma=b^{(1)},~b^{(2)},~ b^{(3)},~\sigma^{(1)},~\sigma^{(2)},~\lambda^{},~l^{}$.
In Section \ref{section-finite-fuel-smooth} we consider the MFG with $\widetilde\gamma$. But for simplicity, we use the notation $\gamma$ instead of $\widetilde\gamma$. Moreover,  in Section \ref{section-finite-fuel-smooth}, we
consider terminal cost $g(X^o_{T+1},\mu_{T+1})$ instead of $g(X_T,\mu_T)$; see Section \ref{section-finite-fuel-smooth} for details.

In order to make MFG in Section \ref{section-finite-fuel-smooth} converge to the original MFG \eqref{general-MFG},  in Section \ref{sec:approximation} we make a further assumption that the coefficients are trivially extended from $[0,T]$ to $[0,T+1]$, i.e.,

\vspace{-6.6mm}
 		\begin{equation}\label{trivial-extension}
 		\widetilde \gamma(t,\cdot)=
 		\overline\gamma(t)\gamma(t,\cdot),
 		\end{equation}
 where $\overline\gamma(t)=1$ when $0\leq t\leq T$ and $\overline\gamma(t)=0$ elsewhere. In particular, \eqref{trivial-extension} implies \eqref{extension}. 
Again, to simplify the notation, we identify $\widetilde \gamma$ with $\gamma$ in Section \ref{sec:approximation}. 

\vspace{-5.6mm}
\subsection{Existence  of Equilibria with $Z^{[k]}$}\label{section-finite-fuel-smooth}

\vspace{-3.6mm}
In this part, we replace $Z$ by $Z^{[k]}$. Due to the continuity of $Z^{[k]}$, the corresponding MFG becomes
\begin{equation}\label{MFG-smooth}
\left\{ \begin{split}
1. &~\textrm{For fixed }\mu\in\left(\mathcal P_p(\widetilde{\mathcal A}^{m,c}_{0,T+1})\right)^2\times \left(\mathcal P_p(\widetilde{\mathcal D}^{}_{0,T+1})\right)^3,\textrm{ minimize }
J(Z^{[k]};\mu)=\\
&~\mathbb E\left[\sum_{i=1,2}\int_0^{T+1}h(X^{(i)}_s)\cdot\,d(\kappa^{(i)}\overline\mu^{(i)}_s+\eta^{(i)}{Z_s^{(i),[k]}})\right.
\left.+\int_0^{T}f(t,X_t,\mu_t,u_t)\,dt+g(X_{T+1},\mu_{T+1})\right],\\
&~\textrm{such that} \textrm{ for }t\in[0,T+1]\\
&~X^{(i)}_t=\int_0^tb^{(i)}(s,X^{(i)}_s,\mu^{(i)}_s)\,ds+\kappa^{(i)}\overline\mu^{(i)}_t+\eta^{(i)}{Z^{(i),[k]}_t}+\int_0^t\sigma^{(i)}_s\,dW^{(i)}_s,~i=1,2,\\
&~X^{(3)}_t=\int_0^tb^{(3)}(s,X^{(3)}_s,u_s)\,ds+\alpha^{(1)} Z^{(1),[k]}_t-\alpha^{(2)}Z^{(2),[k]}_t+\int_0^tl(s,u_s)\widetilde N(ds).\\
2.&~\textrm{Let }Z\textrm{ and }X\textrm{ be the optimal control and state from 1 and search for the fixed point }\\
&~\mu=(\mathbb P\circ (Z^{(1),[k]})^{-1},\mathbb P\circ (Z^{(2),[k]})^{-1},\mathbb P\circ (X^{(1)})^{-1},\mathbb P\circ (X^{(2)})^{-1},\mathbb P\circ (X^{(3)})^{-1}).
\end{split}\right.
\end{equation}
Here $\widetilde{\mathcal A}^{m,c}_{0,T+1}$ is the set of all elements in $\widetilde{\mathcal A}^{m}_{0,T+1}$ with continuous trajectories. 
Denote by $\mathcal R^{m,[k]}(\mu)$ and $\mathcal R^{m,[k],*}(\mu)$ the set of all relaxed controls and optimal relaxed controls corresponding to \eqref{MFG-smooth}, respectively and $\mathbb P\in\mathcal R^{m,[k]}(\mu)$ if and only if it is a probability measure supported on $\Omega^{m,o}$ and it satisfies Definition \ref{control-rule-mu} with Item $2.$ modified as Item $2'.$:

\vspace{-2.1mm}
\textbf{$2'$.} 
there exists an adapted process $Y\in\widetilde{\mathcal D}_{0,T+1}$ such that

\vspace{-8.6mm}
\begin{equation*}
\begin{split}
1)&~\mathbb P\left(Y=X^{o,(3)}-\alpha^{(1)} Z^{(1),[k]}+\alpha^{(2)}Z^{(2),[k]}\right)=1, \textrm{ and for each }\phi\in \mathcal{C}^2_b(\mathbb{R}^d;\mathbb R), \textrm{ it holds that }\\
2)&~\left(\mathcal M^{\phi,X^{o,(1)},Z^{(1)},\mu^{(1)},[k]}_t\right)_{0\leq t\leq T+1} \textrm{ and }\left(\mathcal M^{\phi,X^{o,(2)},Z^{(2)},\mu^{(2)},[k]}_t\right)_{0\leq t\leq T+1}\textrm{ are continuous }\mathbb P \textrm{ martingales}, \\
3)&~\left(\mathcal M^{\phi,X^{o,(3)},Y^{},Q,[k]}_t\right)_{0\leq t\leq T+1}\textrm{ is a }\mathbb P\textrm{ martingale with c\`adl\`ag path},
\end{split}
\end{equation*}
where for $t\in[0,T+1]$ and $i=1,2$, $	\mathcal M^{\phi,X^{o,(i)},Z^{(i)},\mu^{(i)},[k]}_t$ is defined as

\vspace{-8.6mm}
 \begin{equation}\label{martingale-def-smooth-12}
	\begin{split}
	&~\phi(X^{o,(i)}_t)-\int_0^t\mathbb L^{(i)}\phi(s,X_s^{o,(i)})\,ds-\int_0^t\nabla\phi(X^{o,(i)}_s)\cdot\,d(\kappa^{(i)}\overline\mu^{(i)}_s+\eta^{(i)}Z^{(i),[k]}_s),
	\end{split}
	\end{equation}

\vspace{-5.6mm}
	and $	\mathcal{M}^{\phi,X^{o,(3)},Y^{},Q}_t$ is defined as
	
	\vspace{-10.6mm}
	\begin{equation}\label{martingale-def-smooth}
	\begin{split}
&~\phi(Y^{}_t)-\int_0^t\int_U\mathcal L^{}\phi(s,X^{o,(3)}_s,Y^{}_s,u)\,Q_s(du)ds.
	\end{split}
	\end{equation}

\vspace{-5.1mm}
The cost functional corresponding to $\mathbb P\in\mathcal R^{m,[k]}(\mu)$ is defined as

\vspace{-8.6mm}
\begin{equation}\label{def-Jo}
	\begin{split}
	J^{o}(\mathbb P;\mu)=&~\mathbb E^{\mathbb P}\left[\sum_{i=1,2}\int_0^{T+1}h(X^{o,(i)}_s)\cdot\,d(\kappa^{(i)}\overline\mu^{(i)}_s+\eta^{(i)}{Z_s^{(i),[k]}})\right.\\
	&~\left.+\int_0^{T}\int_Uf(t,X^o_t,\mu_t,u)\,Q_t(du)dt+g(X^o_{T+1},\mu_{T+1})\right].
	\end{split}
\end{equation}
\begin{remark}
By the continuity of $Z^{[n]}$, the result in the current section (Section \ref{section-finite-fuel-smooth}) holds under $J_1$ topology. But in Section \ref{sec:approximation} the convergence from $Z^{[n]}$ to $Z$ only holds under $M_1$ topology. So in Section \ref{section-finite-fuel-smooth} our analysis will be based on $M_1$ topology and the argument will be used in Section \ref{sec:approximation} and Section \ref{sec:general-control}. 
Moreover, we notice that it is unnecessary to extend the integral horizon of $f$ in \eqref{MFG-smooth}.
\end{remark}

\vspace{-1.6mm}
In the current section, we prove the existence of equilibria for \eqref{MFG-smooth} for each fixed $k$. To apply Berge's maximum theorem and Kakutani-Fan-Glicksberg fixed point theorem, we will prove the union of all possible relaxed controls is relatively compact in Lemma \ref{relative-compactness-union-control-rule}, the cost functional is jointly continuous on the graph of $\mathcal R^{m,[k]}$ in Lemma \ref{continuity-J-mu-P} and the graph is closed in Proposition \ref{admissibility-limit-P-proposition}, where by graph we mean
\[
\textrm{Gr}\mathcal{R}^{m,[k]} :=
\left\{(\mu,\mathbb{P})\in\left(\mathcal{P}_p(\widetilde{\mathcal{A}}^{m,c}_{0,T+1})\right)^2\times\left(\mathcal{P}_p(\widetilde{\mathcal{D}}_{0,T+1})\right)^3\times\mathcal{P}_p(\Omega^o):~\mathbb{P}\in\mathcal{R}^{m,[k]}(\mu)\right\}.
\]
\begin{lemma}\label{relative-compactness-union-control-rule}
Under assumptions $\mathcal{A}_1$, $\mathcal{A}_4$ and $\mathcal{A}_6$, the set $\bigcup_{\mu\in(\mathcal{P}_p(\widetilde{\mathcal{A}}^{m,c}_{0,T+1}))^2\times(\mathcal{P}_p(\widetilde{\mathcal{D}}_{0,T+1}))^3}\mathcal{R}^{m,[k]}(\mu)$ is relatively compact in $\mathcal{W}_{p}$, for each fixed $k$.
\end{lemma}

\vspace{-7.6mm}
\begin{proof} 
Let $\{\mu^{n}\}_{n\geq 1}$ be any sequence in $(\mathcal{P}_p(\widetilde{\mathcal{A}}^{m,c}_{0,T+1}))^2\times(\mathcal{P}_p(\widetilde{\mathcal{D}}_{0,T+1}))^3$ and $\mathbb{P}^n\in \mathcal{R}^{[k]}(\mu^{n}), n\geq 1$. Since $U$ and $\widetilde{\mathcal{A}}^m_{0,T}$ are compact by assumption and \cite[Theorem 12.12.2]{Whitt-2002}, respectively, $\{\mathbb{P}^n\circ Q^{-1}\}_{ n\geq 1}$ and $\{\mathbb{P}^n\circ (Z^{i})^{-1}\}_{ n\geq 1}$ are tight, and even relatively compact in the topology induced by Wasserstein metric, since $\widetilde{\mathcal{U}}_{0,T}$ and $\widetilde{\mathcal{A}}^{m}_{0,T}$ are compact.

\vspace{-2.1mm}
By Proposition \ref{martingale-representation} there exist extensions $(\bar{\Omega}^n,\bar{\mathcal{F}}^n,\{\bar{\mathcal{F}}^n_t\},\mathbb{Q}^n)$ of the canonical path space $\Omega^o$ and processes $({X}^{n},{Z}^{n},Q^n,W^n,{N}^n)$ defined on $(\bar{\Omega}^n,\bar{\mathcal{F}}^n,\{\bar{\mathcal{F}}^n_t\},\mathbb{Q}^n)$, such that for $t\in[0,T+1]$
    \[
        d{X}^{(i),n}_t=b^{(i)}(t,{X}^{(i),n}_t,\mu^{(i),n}_t)\,dt+\,d(\kappa^{(i)}\overline\mu^{(i),n}_t+\eta^{(i)}Z^{(i),n,[k]}_t)+\sigma^{(i)}(t)\,dW^{(i),n}_t,\quad i=1,2,
    \]
    \begin{equation*}\label{expression-X3-lemma2.2}
    	\begin{split}
        dX^{(3),n}_t=&~\int_U b^{(3)}(t,X^{(3),n}_t,u)\,Q^n_t(du)dt+\alpha^{(1)}\,dZ^{(1),n,[k]}_t-\alpha^{(2)}\,dZ^{(2),n,[k]}_t+\int_Ul^{}(t,u)\widetilde N^{n}(dt,du),
   		\end{split}
    \end{equation*}
    and
   $
        \mathbb{P}^n=\mathbb{P}^n\circ (X^o,Q,Z)^{-1}=\mathbb{Q}^n\circ(X^n,Q^n,Z^{n})^{-1},
    $
where $Z^{n,[k]}_t=k\int_{t-1/k}^tZ_s^{n}\,ds$ and $N^n$ is a Poisson random measure on $[0,T]\times U$ with intensity $Q^n_t(du)\lambda_t\,dt$.
Thus, the relative compactness of $\{\mathbb{P}^n\circ (X^{o,(i)})^{-1}\}_{n\geq 1}$ is equivalent to relative compactness of $\{\mathbb{Q}^n\circ(X^{(i),n})^{-1}\}_{n\geq 1}$. By assumption $\mathcal A_1$, $\mathcal A_4$ and the boundedness of singular controls, we have for any $p\geq 1$

\vspace{-9.6mm}
\begin{equation}\label{widetilde-C-1}
    \begin{split}
        \sup_{i=1,2,3}\sup_n\mathbb E^{\mathbb Q^n}\left[\sup_{0\leq t\leq T}\|X^{(i),n}_t\|^{  p}\right]\leq \widetilde C_1<\infty;
    \end{split}
\end{equation}

\vspace{-5.2mm}
Moreover, by the monotonicity of $\kappa^{(i)}\overline\mu^{(i),n}+\eta^{(i)}Z^{(i),n}$, we have for any $t_1<t_2<t_3$ and for $i=1,2$ that
    $\max_{1\leq j\leq d}\left|X^{(i),n,j}_{t_2}-[ X^{(i),n,j}_{t_1},X^{(i),n,j}_{t_3}]\right|
    =\max_{1\leq j\leq d}
    \inf_{\lambda\in[0,1]}\left|X^{(i),n,j}_{t_2}-\lambda X^{(i),n,j}_{t_1}-(1-\lambda)X^{(i),n,j}_{t_3}\right|$ $
    \leq
    \left\|\int_{t_1}^{t_2}b^{(i)}(t,X^{(i),n}_t)\,dt\right\|+\left\|\int_{t_2}^{t_3}b^{(i)}(t,X^{(i),n}_t)\,dt\right\|+\left\|\int_{t_1}^{t_2}\sigma^{(i)}(s)\,dW^{(i),n}_s\right\|+\left\|\int_{t_2}^{t_3}\sigma^{(i)}(s)\,dW^{(i),n}_s\right\|,$
which implies the existence of $k(\delta)$ with $\lim_{\delta\rightarrow 0}k(\delta)=0$ such that
\begin{equation}\label{k-delta}
    \begin{split}
        \mathbb Q^n(\widetilde w(X^{(i),n},\delta)\geq \eta)\leq \frac{\mathbb E^{\mathbb Q^n}[\widetilde w(X^{(i),n},\delta)]}{\eta}\leq\frac{k(\delta)}{\eta},
    \end{split}
\end{equation}
where $\widetilde w$ is the extended oscillation function of $M_1$ topology; see \cite[Appendix B]{FH-2017}.

Finally, by the linear growth of $b^{(3)}$, boundedness of $l$, and compactness of $U$ and $\widetilde{\mathcal A}^m_{0,T}$, and the uniform bound \eqref{widetilde-C-1}, it holds that
\begin{equation}\label{widetilde-C-2}
\sup_n\sup_{\tau}\mathbb E^{\mathbb Q^n}\|X^{(3),n}_{\tau+\delta}-X^{(3),n}_\tau\|^2\leq \widetilde C_2\delta,
\end{equation}
where $\tau$ is the stopping time taking values in $[0,T+1]$.
Thus, Aldous's tightness criterion (\cite[Theorem 16.10]{Billingsley-1968}) implies that tightness of $\mathbb Q^n\circ (X^{(3),n})^{-1}$ in $J_1$ topology thus in $M_1$ topology.
\end{proof}
\begin{lemma}\label{continuity-J-mu-P}
Let assumptions $\mathcal{A}_1$-$\mathcal A_6$ hold. Then
$J^o:\textrm{Gr}\mathcal{R}^{m,[k]}\rightarrow\mathbb{R}$ is continuous.
\end{lemma}

\vspace{-5.6mm}
\begin{proof}
By Proposition \ref{martingale-representation} and Lemma \ref{lem-new-cost} we have
\begin{align*}
        J^o(\mathbb P;\mu)=&~\mathbb E^{\mathbb P}\left[\sum_{j=1}^d\int_{X^{o,(1)}_{j,0}}^{X_{j,T+1}^{o,(1)}}h_j(x)\,dx-\int_0^{T+1}h(X^{o,(1)}_t)\cdot b^{(1)}(t,X^{o,(1)}_t)\,dt-\frac{1}{2}\sum_{j=1}^d\int_0^{T+1}a^{(1)}_{jj}(t)h'_j(X^{o,(1)}_{j,t})\,dt\right.\\
        &~+\sum_{j=1}^d\int_{X^{o,(2)}_{j,0}}^{X_{j,T+1}^{o,(2)}}h_j(x)\,dx-\int_0^{T+1}h(X^{o,(2)}_t)\cdot b^{(2)}(t,X^{o,(2)}_t)\,dt-\frac{1}{2}\sum_{j=1}^d\int_0^{T+1}
        a^{(2)}_{jj}(t)h'_j(X^{o,(2)}_{j,t})\,dt\\
        &~\left.+\int_0^T\int_Uf(t,X^o_t,\mu_t,u)\,Q_t(du)dt+g(X^o_{T+1},\mu_{T+1})\right].
\end{align*}
By \cite[Theorem 12.5.2]{Whitt-2002}, $x^n\rightarrow x$ in $(\widetilde{\mathcal D}_{0,T+1}(\mathbb R;\mathbb R^d), M_1)$ is equivalent to $x^n_j\rightarrow x_j$ in $(\widetilde{\mathcal D}_{0,T+1}(\mathbb R;\mathbb R),M_1)$ for each $j=1,\cdots,d$. Then the joint continuity can be verified by the same argument as that in the proof of \cite[Lemma 3.3]{FH-2017}.
\end{proof}

\begin{proposition}\label{admissibility-limit-P-proposition}
The assumptions $\mathcal{A}_1$, $\mathcal{A}_4$ and $\mathcal{A}_6$ imply the set-valued map $\mathcal R^{m,[k]}$ has a closed graph, i.e.,
for any sequence $\{\mu^{n}\}_{n\geq 1}\subseteq(\mathcal{P}_p(\widetilde{\mathcal{A}}^{m,c}_{0,T+1}))^2\times( \mathcal{P}_p(\widetilde{\mathcal{D}}_{0,T+1}))^3$ and $\mu\in(\mathcal{P}_p(\widetilde{\mathcal{A}}^{m,c}_{0,T+1}))^2\times( \mathcal{P}_p(\widetilde{\mathcal{D}}_{0,T+1}))^3$ with $\mu^{n}\rightarrow\mu^{}$ in $\mathcal{W}_{p,\widetilde{\mathcal{A}}^{m,c}_{0,T+1}}\times \mathcal{W}_{p,\widetilde{\mathcal{A}}^{m,c}_{0,T+1}}\times \mathcal{W}_{p,\widetilde{\mathcal{D}}_{0,T+1}}\times \mathcal{W}_{p,\widetilde{\mathcal{D}}_{0,T+1}}\times \mathcal{W}_{p,\widetilde{\mathcal{D}}_{0,T+1}}$, if $\mathbb{P}^n\in\mathcal{R}^{m,[k]}(\mu^{n})$ and $\mathbb{P}^n\rightarrow\mathbb{P}$ in $\mathcal{W}_{p,\Omega^{m,o}}$, then $\mathbb{P}\in\mathcal{R}^{m,[k]}(\mu)$.
\end{proposition}

\vspace{-4.6mm}
\textit{Proof.}
To verify $\mathbb P\in\mathcal R^{m,[k]}(\mu)$, it suffices to check the items in the definition of relaxed controls. For each $n$, there exists a stochastic process $Y^n\in\widetilde{\mathcal{D}}_{0,T+1}$ such that $\mathbb{P}^n\left(X^{o,(3)}_\cdot=Y^{n}_\cdot+\alpha^{(1)}Z^{(1),[k]}_\cdot-\alpha^{(2)}Z^{(2),[k]}_\cdot\right)=1$ and the corresponding martingale problem is satisfied. By Proposition \ref{martingale-representation}, for each $n$ there exists a probability space $(\Omega^n,\mathcal{F}^n,\mathbb{Q}^n)$ that accommodates $(\check{X}^{n},\check{Q}^n,\check{Z}^{n})$, a Poisson random measure $N^{n}$ with intensity $\check Q^n_t(du)\lambda^{}_tdt$, and two Brownian motions $W^{(1),n}$ and $W^{(2),n}$ such that 
$	\mathbb{P}^n\circ(X^o,Q,Z,Y^n)^{-1}=\mathbb{Q}^n\circ (\check{X}^{n},\check{Q}^n,\check{Z}^{n},\check Y^n)^{-1},$
where

\vspace{-5.8mm}
\begin{equation*}
	\left\{\begin{split}
	&~\check Y^{n}_\cdot=\int_0^\cdot\int_U b^{(3)}(s,\check{X}^{(3),n}_s,u)\,\check{Q}^n_s(du)ds+\int_0^\cdot\int_Ul^{(1)}(s,u)\widetilde N^{n}(ds,du),\\
	&~\check{X}^{(3),n}_\cdot=\check Y^{n}_\cdot+\alpha^{(1)}\check{Z}^{(1),n,[k]}_\cdot-\alpha^{(2)}\check Z^{(2),n,[k]}_\cdot,\\
	&~\check X^{(i),n}_\cdot=\int_0^\cdot b^{(i)}(s,\check X^{(i),n}_s,\mu^{(i),n}_s)\,ds+\kappa^{(i)}\overline\mu^{(i),n}_\cdot+\eta^{(i)}\check Z^{(i),n,[k]}_\cdot+\int_0^\cdot\sigma^{(i)}(s)\,dW^{(i),n}_s,\quad i=1,2.
	\end{split}\right.
\end{equation*}
The relative compactness of $\check Y^n$ (thus the relative compactness of $Y^n$) follows from the same argument as Lemma \ref{relative-compactness-union-control-rule}. As a result, the sequence $(X^{o},Q,Z,Y^n)$ of random variables taking values in ${\Omega}^{m,o} \times \widetilde{\cal D}_{0,T+1}$ has a weak limit $(\widehat{X},\widehat{Q},\widehat{Z},\widehat{Y})$ defined on some probability space.
Skorokhod's representation theorem yields a probability space $(\widetilde{\Omega},\widetilde{\mathcal{F}},\mathbb{Q})$ that accommodates  $(\widetilde{X}^{n},\widetilde{Q}^n,\widetilde{Z}^{n},\widetilde{Y}^n)$ and $(\widetilde{X}^{},\widetilde{Q},\widetilde{Z},\widetilde{Y})$ such that
\begin{equation}\label{admissibility-identity-law-2}
    \begin{split}
    &~(\widetilde{X}^{n},\widetilde{Q}^n,\widetilde{Z}^{n},\widetilde{Y}^n)\overset{d}{=}(X^{o},Q,Z^{},Y^n),\quad\quad (\widetilde{X}^{},\widetilde{Q},\widetilde{Z}^{},\widetilde{Y})\overset{d}{=}(\widehat{X}^{},\widehat{Q},\widehat{Z}^{},\widehat{Y}),
    \end{split}
\end{equation}
 and as elements in the product space $\Omega^{m,o}\times \widetilde{\mathcal D}_{0,T+1}$
\begin{equation}\label{admissibility-as-convergence}
	(\widetilde{X}^{n},\widetilde{Q}^n,\widetilde{Z}^{n},\widetilde{Y}^n)\rightarrow(\widetilde{X}^{},\widetilde{Q},\widetilde{Z}^{},\widetilde{Y}) \quad \mathbb{Q}\mbox{-a.s.}.
\end{equation}
In particular,
$
	\mathbb{Q} \left(\widetilde{X}^{(3)}=\widetilde{Y}^{} +\alpha^{(1)}\widetilde{Z}^{(1),[k]}-\alpha^{(2)}\widetilde Z^{(2),[k]}\right)=1.
$
By $\mathbb{P}^n\rightarrow \mathbb{P}$ and the uniqueness of the limit, we have $\mathbb{P}\circ(X^{o},Q,Z^{})^{-1}=\mathbb{Q}\circ(\widetilde{X}^{},\widetilde{Q},\widetilde{Z})^{-1}$. It yields a stochastic process $Y\in\widetilde{\mathcal{D}}_{0,T+1}$ such that
    $\mathbb{P} \left( X^{o,(3)}=Y+\alpha^{(1)}Z^{(1),[k]}-\alpha^{(2)}Z^{(2),[k]}\right)=1,$
 and 
 \begin{equation}\label{admissibility-identity-law-3}
 \mathbb{P}\circ(X^o,Q,Z,Y)^{-1}=\mathbb{Q}\circ(\widetilde{X},\widetilde{Q},\widetilde{Z},\widetilde{Y})^{-1}.
 \end{equation}
It remains to verify $\mathcal M^{\phi,X^{o,(1)},Z^{(1)},\mu^{(1)},[k]}$,  $\mathcal M^{\phi,X^{o,(2)},Z^{(2)},\mu^{(2)},[k]}$ and ${\mathcal{M}}^{\phi,X^{o,(3)},Y,Q}$ are martingales under $\mathbb P$, where $\mathcal M^{\phi,X^{o,(1)},Z^{(1)},\mu^{(1)},[k]}$,  $\mathcal M^{\phi,X^{o,(2)},Z^{(2)},\mu^{(2)},[k]}$ and ${\mathcal{M}}^{\phi,X^{o,(3)},Y,Q}$ are defined in \eqref{martingale-def-smooth-12} and \eqref{martingale-def-smooth}, respectively.
First, we verify the martingale property related to \eqref{martingale-def-smooth-12}. Note that $\mu^n\rightarrow\mu$ in $(\mathcal P_p(\widetilde{\mathcal A}^{m,c}_{0,T+1}))^2\times( \mathcal P_p(\widetilde{\mathcal D}_{0,T+1}))^3$ implies
 \begin{equation}\label{convergence-1st-moment-smooth}
	\overline\mu^{(i),n}\rightarrow\overline\mu^{(i)}\quad\textrm{    in uniform topology},\quad i=1,2.
\end{equation}
 Thus, we have for any $s<t$ and any $\mathcal F^o_s$-measurable continuous and bounded function $F$ defined on the canonical space $\Omega^{m,o}$
\begin{align*}
		0=&~\mathbb E^{\mathbb P^n}\left(\mathcal M^{\phi,X^{o,(1)},Z^{(1)},\mu^{(1),n},[k]}_t-\mathcal M^{\phi,X^{o,(1)},Z^{(1)},\mu^{(1),n},[k]}_s\right)F\quad(\textrm{since }\mathbb P^n\in\mathcal R^{[k]}(\mu^n))\\
		\overset{\textrm{by }\eqref{admissibility-identity-law-2}}{=}&~\mathbb E^{\mathbb Q}\left(\mathcal M^{\phi,\widetilde X^{(1),n},\widetilde Z^{(1),n},\mu^{(1),n},[k]}_t-\mathcal M^{\phi,\widetilde X^{(1),n},\widetilde Z^{(1),n},\mu^{(1),n},[k]}_s\right)F(\widetilde X^n,\widetilde Q^n,\widetilde Z^n)\\
		\overset{\textrm{by }\eqref{admissibility-as-convergence}\textrm{ and }\eqref{convergence-1st-moment-smooth}}{\rightarrow}&~\mathbb E^{\mathbb Q}\left(\mathcal M^{\phi,\widetilde X^{(1)},\widetilde Z^{(1)},\mu^{(1)},[k]}_t-\mathcal M^{\phi,\widetilde X^{(1)},\widetilde Z^{(1)},\mu^{(1)},[k]}_s\right)F(\widetilde X,\widetilde Q,\widetilde Z)\\
		\overset{\textrm{by }\eqref{admissibility-identity-law-3}}{=}&~\mathbb E^{\mathbb P}\left(\mathcal M^{\phi,X^{o,(1)},Z^{(1)},\mu^{(1)},[k]}_t-\mathcal M^{\phi,X^{o,(1)},Z^{(1)},\mu^{(1)},[k]}_s\right)F.
\end{align*}
The same result holds for $\mathcal M^{\phi,X^{o,(2)},Z^{(2)},\mu^{(2)},[k]}$.

Next we check the martingale property of ${\mathcal{M}}^{\phi,X^{o,(3)},Y,Q}$.
Since $\widetilde Y^n\rightarrow\widetilde Y$ in $M_1$ topology $\mathbb Q$ a.s., there exists $\widetilde\Omega'\subseteq\widetilde\Omega$ with full measure such that for each $\widetilde\omega\in\widetilde\Omega'$, $\widetilde Y^n_t(\widetilde\omega)\rightarrow\widetilde Y_t(\widetilde\omega)$ for almost every $t\in[0,T+1]$, which together with Step 2 in the proof of \cite[Lemma 3.3]{FH-2017} implies that for each $\widetilde\omega\in\widetilde\Omega$ and for each continuous and bounded $F$, 
$
    \lim_{n\rightarrow\infty}\int_0^{T+1}\left|{\mathcal M}^{\phi,\widetilde X^{(3),n},\widetilde Y^n,\widetilde Q^n}_tF(\widetilde X^n,\widetilde Q^n,\widetilde Z^n)-{\mathcal M}^{\phi,\widetilde X^{(3)},\widetilde Y,\widetilde Q}_tF(\widetilde X,\widetilde Q,\widetilde Z)\right|(\widetilde\omega)\,dt=0.
$
By the dominated convergence, it holds that
\[
    \lim_{n\rightarrow\infty}\mathbb E^{\mathbb Q}\Big[\int_0^{T+1}\Big|{\mathcal M}^{\phi,\widetilde X^{(3),n},\widetilde Y^n,\widetilde Q^n}_tF(\widetilde X^n,\widetilde Q^n,\widetilde Z^n)-{\mathcal M}^{\phi,\widetilde X^{(3)},\widetilde Y,\widetilde Q}_tF(\widetilde X,\widetilde Q,\widetilde Z)\Big|\,dt\Big]=0.
\]
Thus, up to a subsequence, we have for almost every $t\in[0,T+1]$ that
\begin{equation}\label{step-M3-martingale}
    \lim_{n\rightarrow\infty}\mathbb E^{\mathbb Q}\left[{\mathcal M}^{\phi,\widetilde X^{(3),n},\widetilde Y^n,\widetilde Q^n}_tF(\widetilde X^n,\widetilde Q^n,\widetilde Z^n)-{\mathcal M}^{\phi,\widetilde X^{(3)},\widetilde Y,\widetilde Q}_tF(\widetilde X,\widetilde Q,\widetilde Z)\right]=0,
\end{equation}
which implies that for almost every $s,t\in[0,T+1)$ and $s<t$, and for each $F$ that is continuous, bounded and $\mathcal{F}^o_s$-measurable
\begin{equation}\label{martingale-convergence}
        \begin{split}
            0=&~\mathbb E^{\mathbb{P}^n}\left[\left({\mathcal{M}}^{\phi,X^{o,(3)},Y^n,Q}_t-{\mathcal{M}}^{\phi,X^{o,(3)},Y^n,Q}_s\right)F(X^o,Q,Z)\right]\quad(\textrm{since }\mathbb P^n\in\mathcal R^{[k]}(\mu^n))\\
            \overset{\textrm{by }\eqref{admissibility-identity-law-2}}{=}&~\mathbb E^{\mathbb{Q}}\left[\left({\mathcal{M}}^{\phi,\widetilde X^{(3),n},\widetilde Y^n,\widetilde Q^n}_t-{\mathcal{M}}^{\phi,\widetilde X^{(3),n},\widetilde Y^n,\widetilde Q^n}_s\right)F(\widetilde{X}^n,\widetilde{Q}^n,\widetilde{Z}^n)\right]\\
            \overset{\textrm{by }\eqref{step-M3-martingale}}{\rightarrow}&~ \mathbb E^{\mathbb{Q}}\left[\left({\mathcal{M}}^{\phi,\widetilde X^{(3)},\widetilde Y,\widetilde Q}_t-{\mathcal{M}}^{\phi,\widetilde X^{(3)},\widetilde Y,\widetilde Q}_s\right)F(\widetilde{X},\widetilde{Q},\widetilde{Z})\right]\\
            \overset{\textrm{by }\eqref{admissibility-identity-law-3}}{=}&~\mathbb E^{\mathbb{P}}\left[\left({\mathcal{M}}^{\phi,X^{o,(3)},Y,Q}_t-{\mathcal{M}}^{\phi,X^{o,(3)},Y,Q}_s\right)F(X^o,Q,Z)\right].
        \end{split}
    \end{equation}
The convergence in \eqref{martingale-convergence} is still true for $t=T+1$. Indeed, the same argument in the proof of \cite[Lemma 3.3]{FH-2017} implies
$
	\int_0^{T+1}\int_U\mathcal L\phi(s,\widetilde X^{(3),n}_s,\widetilde Y^n_s,u)\widetilde Q^n_s(du)ds\rightarrow 	\int_0^{T+1}\int_U\mathcal L\phi(s,\widetilde X^{(3)}_s,\widetilde Y_s,u)\widetilde Q_s(du)ds$ $\mathbb Q~\textrm{a.s.}.
$
Note that $\widetilde Y^n\rightarrow\widetilde Y$ in $(\widetilde{\mathcal D}_{0,T+1},M_1)$ $\mathbb Q$ a.s. implies that
$
    \widetilde Y^n_{T+1}\rightarrow\widetilde Y_{T+1},~\mathbb Q
$ a.s..
Thus, by dominated convergence it holds that
$
    \mathbb E^{\mathbb Q}\left[{\mathcal M}^{\phi,\widetilde X^{(3),n},\widetilde Y^n,\widetilde Q^n}_{T+1}F(\widetilde X^n,\widetilde Q^n,\widetilde Z^n)\right]\rightarrow\mathbb E^{\mathbb Q}\left[{\mathcal M}^{\phi,\widetilde X^{(3)},\widetilde Y,\widetilde Q}_{T+1}F(\widetilde X,\widetilde Q,\widetilde Z)\right].
$
By the right continuity of the trajectory of ${\mathcal M}^{\phi,X^{o,(3)},Y,Q}$, we have for any $0\leq s<t\leq T+1$
\[
~~~\qquad\qquad \qquad\qquad\qquad  \mathbb E^{\mathbb{P}}\left[\left({\mathcal{M}}^{\phi,X^{o,(3)},Y,Q}_t-{\mathcal{M}}^{\phi,X^{o,(3)},Y,Q}_s\right)F(X^o,Q,Z)\right]=0.   \qquad\qquad\qquad\qquad \qquad  \square 
\]
\begin{corollary}\label{continuity-R}
Suppose that $\mathcal{A}_1$, $\mathcal{A}_4$ and $\mathcal{A}_6$ hold. Then,  $\mathcal{R}^{m,[k]}:(\mathcal{P}_p(\widetilde{\mathcal{A}}^{m,c}_{0,T+1}))^2\times ( \mathcal{P}_p(\widetilde{\mathcal{D}}_{0,T+1}))^3\rightarrow 2^{\mathcal{P}_p(\Omega^{m,o})}$ is continuous in the sense of \cite[Definition 17.2, Theorem 17.20, Theorem 17.21]{AB-1999} and compact-valued.
\end{corollary}

\vspace{-8.6mm}
\begin{proof}
Lemma \ref{relative-compactness-union-control-rule}, Proposition \ref{admissibility-limit-P-proposition} and \cite[Theorem 17.20]{AB-1999} imply that $\mathcal R^{m,[k]}$ is upper hemi-continuous and compact-valued. The lower hemi-continuity of $\mathcal{R}^{m,[k]}$ can be verified in the same manner as \cite[Lemma 4.4]{L-2015}. 
%
\end{proof}


\begin{corollary}\label{existence-singular-control}
Under assumptions $\mathcal{A}_1$-$\mathcal{A}_6$, $\mathcal{R}^{m,[k],*}(\mu)\neq{\O}$ for each $\mu\in(\mathcal{P}_p(\widetilde{\mathcal{A}}^{m,c}_{0,T+1}))^2\times (\mathcal{P}_p(\widetilde{\mathcal{D}}_{0,T+1}))^3$ and $\mathcal{R}^{m,[k],*}$ is upper hemi-continuous.
\end{corollary}

\vspace{-8.6mm}
\begin{proof}
By Lemma \ref{continuity-J-mu-P} and Corollary \ref{continuity-R}, the conditions in Berge's maximum theorem (see \cite[Theorem 17.31]{AB-1999}) are satisfied. Thus, the desired results follow.
\end{proof}

\begin{theorem}\label{existence-MFG-finite-fuel}
Under assumptions $\mathcal{A}_1$-$\mathcal{A}_6$ and the finite-fuel constraint $Z \in \widetilde{\mathcal{A}}^m_{0,T}\times  \widetilde{\mathcal{A}}^m_{0,T}$, there exists a relaxed solution to (\ref{MFG-smooth}).
\end{theorem}

\vspace{-7.6mm}
\begin{proof}
Define a set-valued map $\psi$ by
 \begin{align*}\label{set-valued-function}
    \psi:~~&(\mathcal{P}_p(\widetilde{\mathcal{A}}^{m,c}_{0,T+1}))^2\times( \mathcal{P}_p(\widetilde{\mathcal{D}}_{0,T+1}))^3 \rightarrow 2^{(\mathcal{P}_p(\widetilde{\mathcal{A}}^{m,c}_{0,T+1}))^2\times( \mathcal{P}_p(\widetilde{\mathcal{D}}_{0,T+1}))^3},\nonumber\\
    &\mu \mapsto \left\{\left(\mathbb{P}\circ (Z^{(1),[k]})^{-1},\mathbb{P}\circ (Z^{(2),[k]})^{-1},\mathbb P\circ (X^{o,(1)})^{-1},\mathbb P\circ (X^{o,(2)})^{-1},\mathbb P\circ (X^{o,(3)})^{-1}\right): \mathbb{P}\in\mathcal{R}^{m,[k],*}(\mu)\right\}.
 \end{align*}
Let $S_1$, $S_2$ and $S_3$ be defined as 
\begin{equation*}
	S_i=\{\mathbb P\circ(X^{o,(i)})^{-1} \in\mathcal P_p(\widetilde{\mathcal D}_{0,T+1}):  ~\mathbb P\in\mathcal P_p(\Omega^{m,o}),~\mathbb P(\widetilde\omega(X^{o,(i)},\delta)\geq \eta)\leq \frac{k(\delta)}{\eta},~\mathbb E^{\mathbb P}\left[ \|X^{o,(i)}\|^p_{T+1}\right]\leq \widetilde C_1  \},~i=1,2
\end{equation*}
and
\begin{equation*}
	\begin{split}
	S_3=&~\left\{\mathbb P\circ(X^{o,(3)})^{-1}\in\mathcal P_p(\widetilde{\mathcal D}_{0,T+1}):\mathbb P\in\mathcal P_p(\Omega^{m,o}),\right.\\
	&~\left.\textrm{for any }X^{(3)}-\textrm{stopping time }\tau,~\mathbb E^{\mathbb P}\left[\|X^{o,(3)}_{\tau+\delta}-X^{o,(3)}_\tau\|^2\right]\leq\widetilde C_2\delta,~\mathbb E^{\mathbb P}\left[\|X^{o,(3)}\|^{p}_{T+1}\right]\leq \widetilde C_1 \right \},
	\end{split}
\end{equation*}
where $k(\delta)$, $\widetilde C_1$ and $\widetilde C_2$ are the given by \eqref{k-delta}, \eqref{widetilde-C-1} and \eqref{widetilde-C-2}, respectively, and a non-negative random variable is called a $X^{o,(3)}$-stopping time if $\{\tau\leq t\}\in\sigma(X^{o,(3)}_t,s\leq t)$ for each $t$. By \cite[Theorem 16.10]{Billingsley-1968}, $S_3$ is relatively compact.

Denote by $\bar S_i$ the closure of $S_i$, $i=1,2,3$. Clearly, $\psi$ maps $(\mathcal{P}_p(\widetilde{\mathcal{A}}^{m,c}_{0,T+1}))^2\times \bar S_1\times\bar S_2\times \bar S_3$ into the power set of itself and $(\mathcal{P}_p(\widetilde{\mathcal{A}}^{m,c}_{0,T+1}))^2\times \bar S_1\times\bar S_2\times\bar S_3$ is non-empty, compact and convex. Moreover, by Corollary \ref{existence-singular-control}, $\psi$ is nonempty-valued and upper hemi-continuous. Indeed, the non-emptiness is obvious and to check the upper hemi-continuity we take any $\mu^n\rightarrow\mu$ in $(\mathcal{P}_p(\widetilde{\mathcal{A}}^{m,c}_{0,T+1}))^2\times \bar S_1\times\bar S_2\times\bar S_3$, Corollary \ref{existence-singular-control} implies the existence of subsequence $\mathbb P^{n_j}\in\mathcal R^{m,[k],*}(\mu^{n_j})$ such that $\mathbb P^{n_j}\rightarrow\mathbb P\in\mathcal R^{[k],*}(\mu)$, which implies $\mathbb P^{n_j} \circ(X^{o,(i)})^{-1}\rightarrow\mathbb P\circ(X^{o,(i)})^{-1}$, $i=1,2,3$. Skorokhod representation implies the existence of $(\mathbb Q,\widehat{\mathcal F},\widehat{\mathcal F}_t)$ and $R^{n_j}:=(R^{(1),n_j},R^{(2),n_j})$ and $R:=(R^{(1)},R^{(2)})$ defined on it such that for $i=1,2$ it holds that $\mathbb P^{n_j}\circ (Z^{(i)})^{-1}=\mathbb Q\circ (R^{(i),n_j})^{-1}$, $\mathbb P\circ (Z^{(i)})^{-1}=\mathbb Q\circ (R^{(i)})^{-1}$ and $R^{n_j}\rightarrow R$ in $(\widetilde{\mathcal A}^m_{0,T},M_1)\times(\widetilde{\mathcal A}^m_{0,T},M_1)$ $\mathbb Q$ a.s., which implies $R^{n_j,[k]}\rightarrow R^{[k]}$ in $(\widetilde{\mathcal A}^m_{0,T+1},M_1)\times(\widetilde{\mathcal A}^m_{0,T+1},M_1)$ $\mathbb Q$ a.s. by \cite[Theorem 12.5.2(iii)]{Whitt-2002}, where $ R^{n_j,[k]}_t:=k\int_{t-1/k}^tR^{n_j}_s\,ds$ and $ R^{[k]}_t:=k\int_{t-1/k}^tR_s\,ds$. Since both $ R^{n_j,[k]}$ and $R^{[k]}$ are continuous, it holds that $R^{n_j,[k]}\rightarrow R^{[k]}$ in uniform topology $\mathbb Q$ a.s. by \cite[Theorem 12.5.2(iv)]{Whitt-2002}. Thus, for any continuous function $\phi$ defined on $\widetilde{\mathcal A}^{m,c}_{0,T+1}$ with $|\phi(y)|\leq C(1+\|y\|_{T+1}^p)$, there holds for $i=1,2$ by dominated convergence
\begin{equation*}
	\begin{split}
		&~\int_{\widetilde{\mathcal A}^{m,c}_{0,T+1}}\phi(y)\mathbb P^{n_j}\circ(Z^{(i),[k]})^{-1}(dy)=\mathbb E^{\mathbb P^{n_j}}\phi(Z^{(i),[k]})=\mathbb E^{\mathbb Q}\phi(R^{(i),n_j,[k]})\\
		\rightarrow&~ \mathbb E^{\mathbb Q}\phi(R^{(i),[k]})=\mathbb E^{\mathbb P}\phi(Z^{(i),[k]})=	\int_{\widetilde{\mathcal A}^{m,c}_{0,T+1}}\phi(y)\mathbb P^{}\circ(Z^{(i),[k]})^{-1}(dy),
	\end{split}
\end{equation*}
which implies the upper hemi-continuity of $\psi$.
Therefore, \cite[Corollary 17.55]{AB-1999} is applicable by embedding $(\mathcal{P}_p(\widetilde{\mathcal{A}}^{m,c}_{0,T+1}))^2\times  (\mathcal{P}_p(\widetilde{\mathcal{D}}_{0,T+1}))^3$ into $(\mathcal{M}(\widetilde{\mathcal{C}}^{}_{0,T+1}))^2\times (\mathcal{M}(\widetilde{\mathcal{D}}_{0,T+1}))^3$, the respective product spaces of all bounded signed measures on $\widetilde{\mathcal{C}}^{}_{0,T+1}$ and $\widetilde{\mathcal{D}}^{}_{0,T+1}$ endowed with the weak convergence topology.
\end{proof}

\vspace{-5.6mm}
\subsection{Approximation.}\label{sec:approximation}

\vspace{-3.6mm}
In this section, the extension \eqref{trivial-extension} is valid throughout.  All the limits in this section are taken as $k\rightarrow\infty$.

In Section \ref{section-finite-fuel-smooth}, we have shown for each fixed $k$, there is an equilibrium $\mathbb P^{m,[k],*}\in\mathcal R^{m,[k],*}(\mu^{m,[k],*})$, where $\mu^{m,[k],*}:=(\mu^{(1),m,[k],*},\mu^{(2),m,[k],*},\mu^{(3),m,[k],*},\mu^{(4),m,[k],*},\mu^{(5),m,[k],*}):=(\mathbb P^{m,[k],*}\circ (Z^{(1),[k]})^{-1},\mathbb P^{m,[k],*}\circ (Z^{(2),[k]})^{-1},\mathbb P^{m,[k],*}\circ (X^{o,(1)})^{-1},\mathbb P^{m,[k],*}\circ (X^{o,(2)})^{-1},\mathbb P^{m,[k],*}\circ (X^{o,(3)})^{-1})$. In this section, we establish the existence of equilibria of \eqref{general-MFG} by constucting $\mathbb P^{m,*}\in\mathcal R^{m,*}(\mu^{m,*})$ with $\mu^{m,*}=(\mathbb P^{m,*}\circ (Z^{(1)})^{-1},\mathbb P^{m,*}\circ (Z^{(2)})^{-1},\mathbb P^{m,*}\circ (X^{(3)})^{-1},\mathbb P^{m,*}\circ (X^{(4)})^{-1},\mathbb P^{m,*}\circ (X^{(5)})^{-1})$, from the sequence $\{\mathbb P^{m,[k],*}\}_{k}$.

By Proposition \ref{martingale-representation}, $\mathbb P^{m,[k],*}\in\mathcal R^{m,[k],*}(\mu^{m,[k],*})$ implies the existence of $(\widehat\Omega^k,\widehat{\mathcal F}^k,\widehat{\mathbb P}^k)$ and $(\widehat X^{k},\widehat Q^k,\widehat Z^{k},\widehat W^k,\widehat N^k)$\footnote{Note that $(\widehat\Omega^k,\widehat{\mathcal F}^k,\widehat{\mathbb P}^k)$ and $(\widehat X^{k},\widehat Q^k,\widehat Z^{k},\widehat W^k,\widehat N^k)$ should depend on $m$. Since $m$ is a fixed finite number in this section, we drop this dependence and only keep the dependence on $m$ for the optimal ones, e.g. $\mu^{m,*}$ and $\mathbb P^{m,*}$.} such that
\begin{equation}\label{app-eq-Xhat12}
\widehat X^{(i),k}_t=\int_0^t b^{(i)}(s,\widehat X^{(i),k}_s,\mu^{(i),m,[k],*}_s)\,ds+\kappa^{(i)}\overline\mu^{(i),m,[k],*}_t+\eta^{(i)}\widehat Z^{(i),k,[k]}_t+\int_0^t\sigma^{(i)}(s)\widehat W^{(i),k}_s,~t\in[0,T+1],\quad i=1,2,
\end{equation}
\begin{equation*}
	\begin{split}
\widehat{X}^{(3),k}_t=&~\int_0^t\int_U b^{(3)}(s,\widehat{X}^{(3),k}_s,u)\,\widehat{Q}^k_s(du)ds\\
&~+\int_0^t\int_Ul(s,u)\widetilde{\widehat N^k}(ds,du)+\alpha^{(1)}\widehat{Z}^{(1),k,[k]}_t-\alpha^{(2)}\widehat Z^{(2),k,[k]}_t,\quad t\in[0,T+1]
	\end{split}
\end{equation*}
and
\begin{equation}\label{app-identity-law-4}
\widehat{\mathbb P}^k\circ\left(\widehat X^{k},\widehat Q^k,\widehat Z^{k}\right)^{-1}=\mathbb P^{m,[k],*}\circ \left(X^o,Q,Z\right)^{-1}.
\end{equation}
Let
\begin{equation}\label{step2-hat-Yk}
	\begin{split}
\widehat Y^k_t:=&~\widehat X^{(3),k}_t-\alpha^{(1)} \widehat{Z}^{(1),k,[k]}_t+\alpha^{(2)}\widehat{Z}^{(2),k,[k]}_t\\
=&~\int_0^t\int_U b^{(3)}(s,\widehat{X}^{(3),k}_s,u)\,\widehat{Q}^k_s(du)ds+\int_0^t\int_Ul(s,u)\widetilde{\widehat N^k}(ds,du),~t\in[0,T+1].
	\end{split}
\end{equation}
Thus, we have
\[
\widehat{\mathbb P}^{k}\circ\left(\widehat{Y}^k,\widehat X^{(1),k},\widehat{X}^{(2),k},\widehat Z^{(1),k},\widehat Z^{(2),k},\widehat Q^k\right)^{-1}=\mathbb P^{m,[k],*}\circ\left({Y}^{[k]},X^{o,(1)},{X}^{o,(2)},Z^{(1)},Z^{(2)},Q\right)^{-1},
\]
where $Y^{[k]}$ is a stochastic process such that $\mathbb P^{m,[k],*}\left(Y^{[k]}=X^{o,(3)}-\alpha^{(1)}{Z}^{(1),[k]}+\alpha^{(2)} Z^{(2),[k]}\right)=1$.
The same argument as in Lemma \ref{relative-compactness-union-control-rule} yields the relative compactness of $\widehat{\mathbb P}^k\circ\left(\widehat{Y}^k,\widehat X^{(1),k},\widehat{X}^{(2),k},\widehat Z^{(1),k},\widehat Z^{(2),k},\widehat Q^k\right)^{-1}$, which implies a weak limit $(\check Y,\check X^{(1)},\check{X}^{(2)},\check Z^{(1)},\check Z^{(2)},\check Q)$. Skorokhod representation theorem implies that there exists a probability space $(\mathring\Omega,\mathring{\mathcal F},\mathring{\mathbb Q})$, two sequences of stochastic processes $$(\mathring Y^k,\mathring X^{(1),k},\mathring{X}^{(2),k},\mathring Z^{(1),k},\mathring Z^{(2),k},\mathring Q^k)\textrm{ and }(\mathring Y,\mathring X^{(1)},\mathring{X}^{(2)},\mathring Z^{(1)},\mathring Z^{(2)},\mathring Q)$$ such that
\begin{equation}\label{step2-id-k-bar-hat}
\left(\mathring Y^k,\mathring X^{(1),k},\mathring{X}^{(2),k},\mathring Z^{(1),k},\mathring Z^{(2),k},\mathring Q^k\right)\overset{d}{=}\left(\widehat{Y}^k,\widehat X^{(1),k},\widehat{X}^{(2),k},\widehat Z^{(1),k},\widehat Z^{(2),k},\widehat Q^k\right),
\end{equation}
\[
\left(\mathring Y,\mathring X^{(1)},\mathring{X}^{(2)},\mathring Z^{(1)},\mathring Z^{(2)},\mathring Q\right)\overset{d}{=}\left(\check{Y},\check X^{(1)},\check{X}^{(2)},\check Z^{(1)},\check Z^{(2)},\check Q\right),
\]
and
\begin{equation}\label{step2-a.s.-1}
\left(\mathring Y^k,\mathring X^{(1),k},\mathring{X}^{(2),k},\mathring Z^{(1),k},\mathring Z^{(2),k},\mathring Q^k\right)\rightarrow\left(\mathring Y,\mathring X^{(1)},\mathring{X}^{(2)},\mathring Z^{(1)},\mathring Z^{(2)},\mathring Q\right), ~\mathring{\mathbb Q}\textrm{ a.s.}.
\end{equation}
Let
\begin{equation}\label{step2-bar-X3k-bar-X3}
\mathring X^{(3),k}:=\mathring Y^k+\alpha^{(1)}\mathring Z^{(1),k,[k]}-\alpha^{(2)}\mathring Z^{(2),k,[k]}\textrm{ and }\mathring X^{(3)}:=\mathring Y+\alpha^{(1)}\mathring Z^{(1)}-\alpha^{(2)}\mathring Z^{(2)}.
\end{equation}
 Thus, 
\eqref{step2-a.s.-1} and \eqref{step2-bar-X3k-bar-X3} imply
the following convergence result
\begin{equation}\label{step2-convergence-bar-X3k-to-bar-X-3}
\mathbb E^{\mathring{\mathbb Q}}\left[\int_0^{T+1}\left\|\mathring X^{(3),k}_t-\mathring X^{(3)}_t\right\|^p\,dt\right]\rightarrow 0\quad\textrm{ and }\quad \mathring X^{(3),k}_{T+1}\rightarrow \mathring X^{(3)}_{T+1}\quad\mathring{\mathbb Q}\textrm{ a.s.}.
\end{equation}

Moreover, \eqref{app-identity-law-4}, \eqref{step2-hat-Yk}, \eqref{step2-id-k-bar-hat} and \eqref{step2-bar-X3k-bar-X3}  imply that
\begin{equation}\label{app-identity-law(3)}
\begin{split}
&~\mathring{\mathbb Q}\circ\left(\mathring Y^k,\mathring X^{(1),k},\mathring{X}^{(2),k},\mathring X^{(3),k},\mathring Z^{(1),k},\mathring Z^{(2),k},\mathring Q^k\right)^{-1}\\
=&~\widehat{\mathbb P}^k\circ\left(\widehat Y^k,\widehat X^{(1),k},\widehat{X}^{(2),k},\widehat X^{(3),k},\widehat Z^{(1),k},\widehat Z^{(2),k},\widehat Q^k\right)^{-1}\\
=&~\mathbb P^{m,[k],*}\circ\left({Y}^{[k]},X^{o,(1)},{X}^{o,(2)},X^{o,(3)},Z^{(1)},Z^{(2)},Q\right)^{-1}.
\end{split}
\end{equation}
Define
\begin{equation}\label{def-mu-*}
	\begin{split}
\mu^{m,*}=&~(\mu^{(1),m,*},\mu^{(2),m,*},\mu^{(3),m,*},\mu^{(4),m,*},\mu^{(5),m,*})\\
:=&~\left(\mathring{\mathbb Q}\circ\left(\mathring Z^{(1)}\right)^{-1},\mathring{\mathbb Q}\circ\left(\mathring Z^{(2)}\right)^{-1},\mathring{\mathbb Q}\circ\left(\mathring X^{(1)}\right)^{-1},\mathring{\mathbb Q}\circ\left(\mathring X^{(2)}\right)^{-1},\mathring{\mathbb Q}\circ\left(\mathring X^{(3)}\right)^{-1}\right).
	\end{split}
\end{equation}
and
\begin{equation}\label{def-P-*}
\mathbb P^{m,*}:=\mathring {\mathbb Q}\circ\left(\mathring X^{(1)},\mathring X^{(2)},\mathring X^{(3)},\mathring Q,\mathring Z^{(1)},\mathring Z^{(2)}\right)^{-1}.
\end{equation}

The next lemma shows the admissbility of $\mathbb P^{m,*}$.
\begin{lemma}\label{lem:stability-martingale-2}
Assume assumptions $\mathcal A_1$, $\mathcal A_4$ and $\mathcal A_6$ hold. 
Let $\mu^{m,*}$ and $\mathbb P^{m,*}$ be defined as \eqref{def-mu-*} and \eqref{def-P-*}, respectively. Then we have $\mathbb P^{m,*}\in\mathcal R^m(\mu^{m,*})$. 
\end{lemma}

\vspace{-7.6mm}
\begin{proof}
The proof is split into two steps. In Step 1, we verify $\mathbb P^{m,*}$ is supported on the original canonical space $\Omega^m$ and we recall $\Omega^m=\widetilde{\mathcal D}_{0,T}\times\widetilde{\mathcal D}_{0,T}\times \widetilde{\mathcal D}_{0,T}\times \widetilde{\mathcal U}_{0,T}\times\widetilde{\mathcal A}^m_{0,T}\times \widetilde{\mathcal A}^m_{0,T}$. In Step 2, we verify the martingale properties.

\textbf{Step 1.}  
By the definition of $\mathbb P^{m,[k],*}$ and the equation \eqref{app-identity-law-4},  $\widehat{Z}^k\in\widetilde{\mathcal A}^m_{0,T}\times \widetilde{\mathcal A}^m_{0,T}$. The trivial extension \eqref{trivial-extension} and \eqref{step2-hat-Yk} imply $\widehat Y^k\in\widetilde{\mathcal D}_{0,T}$. Thus, the tuple of stochastic processes $(\widehat{Y}^k,\widehat X^{(1),k},\widehat{X}^{(2),k},\widehat Z^{(1),k},\widehat Z^{(2),k},\widehat Q^k)$ in fact takes values in
the product space $\widetilde{\mathcal D}_{0,T}\times \widetilde{\mathcal D}_{0,T+1}\times \widetilde{\mathcal D}_{0,T+1}\times\widetilde{\mathcal A}^m_{0,T}\times\widetilde{\mathcal A}^m_{0,T}\times\widetilde{\mathcal U}_{0,T}$ a.s.,  so does $(\mathring{Y}^k,\mathring X^{(1),k},\mathring{X}^{(2),k},\mathring Z^{(1),k},\mathring Z^{(2),k},\mathring Q^k)$ by \eqref{step2-id-k-bar-hat}.

Since $\widetilde{\mathcal D}_{0,T}$, $\widetilde{\mathcal A}^m_{0,T}$ and $\widetilde{\mathcal U}_{0,T}$ are closed, the convergence \eqref{step2-a.s.-1} implies $\mathring Z\in\widetilde{\mathcal A}^m_{0,T}\times\widetilde{\mathcal A}^m_{0,T}$, $\mathring Y\in\widetilde{\mathcal D}_{0,T}$ and $\mathring Q\in\widetilde{\mathcal U}_{0,T}$. Thus, \eqref{step2-bar-X3k-bar-X3} implies $\mathring X^{(3),k}\in\widetilde{\mathcal D}_{0,T+1}$ and $\mathring X^{(3)}\in\widetilde{\mathcal D}_{0,T}$. It remains to prove $(\mathring X^{(1)},\mathring X^{(2)})\in\widetilde{\mathcal D}_{0,T}\times \widetilde{\mathcal D}_{0,T}$.
By \eqref{app-eq-Xhat12} and the trivial extension \eqref{trivial-extension} there exists $\widehat K^{(1),k}\in\widetilde{\mathcal C}_{0,T}$ such that $\widehat X^{(1),k}=\widehat K^{(1),k}+\kappa^{(1)}\overline\mu^{(1),m,[k],*}+\eta^{(1)}\widehat Z^{(1),k,[k]}$
and the same argument as Lemma \ref{relative-compactness-union-control-rule} implies the relative compactness of $\widehat{\mathbb P}^{k}\circ(\widehat X^{(1),k},\widehat K^{(1),k},\widehat Z^{(1),k})^{-1}$ with a weak limit denoted by $(X',K',Z')$. Skorokhod representation implies
	\[
	(\widehat X^{(1),k},\widehat K^{(1),k},\widehat Z^{(1),k})\overset{d}{=}(\widehat X^{'(1),k},\widehat K^{(1),k},\widehat Z^{'(1),k})\quad\textrm{and}\quad (\widehat X^{'(1)},\widehat K^{(1)},\widehat Z^{'(1)})\overset{d}{=}(X',K',Z')
	\]
	and $(\widehat X^{'(1),k},\widehat K^{(1),k},\widehat Z^{'(1),k})\rightarrow (\widehat X^{'(1)},\widehat K^{(1)},\widehat Z^{'(1)})$ in $\widetilde{\mathcal D}_{0,T+1}\times\widetilde{\mathcal C}_{0,T}\times\widetilde{\mathcal A}^m_{0,T}$, which together with \eqref{app-eq-Xhat12} yield
	\[
	\widehat X^{'(1)}=\widehat K^{(1)}+\kappa^{(1)}\overline\mu^{(1),m,*}+\eta^{(1)}\widehat Z^{'(1)}.
	\]
	Thus, $\widehat X^{'(1)}\in\widetilde{\mathcal D}_{0,T}$. Note that by the uniqueness of the limit $\widehat X^{'(1)}\overset{d}{=}\mathring X^{(1)}$, which implies $\mathring X^{(1)}\in\widetilde{\mathcal D}_{0,T}$. The same result holds for $\mathring X^{(2)}$.

\textbf{Step 2.}
 In this step, we check $\mathcal M^{\phi,X^{(3)},Y,Q}$ is a $(\mathbb P^*,(\mathcal F_t)_{0\leq t\leq T})$ martingale. The martingale property of $\mathcal M^{\phi,X^{(1)},Z^{(1)},\mu^*}$ and $\mathcal M^{\phi,X^{(2)},Z^{(2)},\mu^*}$ can be obtained similarly.

Boundedness and linear growth of the coefficients, compactness of $U$, \eqref{step2-a.s.-1}, \eqref{step2-convergence-bar-X3k-to-bar-X-3} and dominated convergence yield that for any bounded and continuous $\Phi$, up to a subsequence,
\begin{equation*}\label{app-martingale-X3-convergence-2}
	\begin{split}
	\lim_{k\rightarrow\infty}\mathbb E^{{\mathring Q}}\left[\int_0^{T+1}\int_0^{T+1}\cdots\int_0^{T+1}\right.&~	\left|\mathcal M^{\phi,\mathring X^{(3),k},\mathring Y^k,\mathring Q^k}_t\Phi(\mathring \zeta^{k}_{t_1},\cdots,\mathring \zeta^{k}_{t_n})\right.\\
	&~\left.\left.-\mathcal M^{\phi,\mathring X^{(3)},\mathring Y,\mathring Q}_t\Phi(\mathring \zeta^{}_{t_1},\cdots,\mathring \zeta^{}_{t_n})\right|\,dt_1\cdots d{t_n}dt\right]=0,
\end{split}
\end{equation*}
where
\begin{equation}\label{app-zeta-1}
\mathring\zeta^{k}_\cdot:=(\mathring X^{k}_{\cdot},\mathring Z^{k}_{\cdot})\quad\textrm{ and }\quad
\mathring\zeta^{}_{\cdot}:=(\mathring X^{}_\cdot,\mathring Z^{}_\cdot).  
\end{equation}
It implies up to a subsequence for almost every $(t,t_1,\cdots,t_n)\in[0,T+1]^{n+1}$
\begin{equation}\label{app-martingale-X3-convergence-2}
\begin{split}
\lim_{k\rightarrow\infty}\mathbb E^{\mathring{\mathbb Q}}	\left|\mathcal M^{\phi,\mathring X^{(3),k},\mathring Y^k,\mathring Q^k}_t\Phi(\mathring \zeta^{k}_{t_1},\cdots,\mathring \zeta^{k}_{t_n})\right.\left.-\mathcal M^{\phi,\mathring X^{(3)},\mathring Y,\mathring Q}_t\Phi(\mathring \zeta^{}_{t_1},\cdots,\mathring \zeta^{}_{t_n})\right|=0.
\end{split}
\end{equation}
Thus, for almost every $(s,t,t_1,\cdots,t_n)\in[0,T+1]^{n+2}$ with $(t,t_1,\cdots,t_n)\in[s,T+1]\times[0,s]^{n}$, any continuous and bounded function $\Phi$ on $(\mathbb R^{d}\times\mathbb R^d\times\mathbb R^d\times\mathbb R^d\times\mathbb R^d)^n$ and any continuous and bounded function $\varphi$ which is defined on $\widetilde {\mathcal U}_{0,T}$ and $\mathcal F_s^Q$ measurable we have
\begin{equation*}
	\begin{split}
	0=&~\mathbb E^{{\mathbb P}^{m,[k],*}}\left(\mathcal M^{\phi,X^{o,(3)},Y^{[k]},Q}_t-\mathcal M^{\phi,X^{o,(3)},Y^{[k]},Q}_s\right)\Phi(\zeta^{o}_{t_1},\cdots,\zeta^{o}_{t_n})\varphi(Q)\\
	\overset{\textrm{by }\eqref{app-identity-law(3)}}{=}&~\mathbb E^{\mathring{\mathbb Q}^{}}\left(\mathcal M^{\phi,\mathring X^{(3),k},\mathring Y^{k},\mathring Q^k}_t-\mathcal M^{\phi,\mathring X^{(3),k},\mathring Y^{k},\mathring Q^k}_s\right)\Phi(\mathring \zeta^{k}_{t_1},\cdots,\mathring \zeta^{k}_{t_n})\varphi(\mathring Q^k)\\
	\overset{\textrm{by }\eqref{app-martingale-X3-convergence-2}}{\rightarrow}&~\mathbb E^{\mathring{\mathbb Q}^{}}\left(\mathcal M^{\phi,\mathring X^{(3)},\mathring Y^{},\mathring Q}_t-\mathcal M^{\phi,\mathring X^{(3)},\mathring Y^{},\mathring Q}_s\right)\Phi(\mathring \zeta^{}_{t_1},\cdots,\mathring \zeta^{}_{t_n})\varphi(\mathring Q)\\
	\overset{\textrm{by }\eqref{def-P-*}}{=}&~\mathbb E^{{\mathbb P}^{m,*}}\left(\mathcal M^{\phi,X^{(3)},Y^{},Q}_t-\mathcal M^{\phi,X^{(3)},Y^{}, Q}_s\right)\Phi( \zeta^{}_{t_1},\cdots, \zeta^{}_{t_n})\varphi(Q),
	\end{split}
\end{equation*}
where $\mathring\zeta$ and $\mathring\zeta^k$ are defined as \eqref{app-zeta-1}, and $\zeta^o$ and $\zeta$ are defined as 
$\zeta^o_{\cdot}:=(X^o_{\cdot},Z_\cdot)$ and $\zeta^{}_{\cdot}:=(X^{}_{\cdot},Z^{}_{\cdot}). $ 
 By the right continuity of $\mathcal M^{\phi,X^{(3)},Y,Q}$ and $X^{(3)}$ we have for any $(t,t_1,\cdots,t_n)\in[s,T+1]\times[0,s]^{n}$ 
 \[
 	\mathbb E^{{\mathbb P}^{m,*}}\left(\mathcal M^{\phi,X^{(3)},Y^{},Q}_t-\mathcal M^{\phi,X^{(3)},Y^{}, Q}_s\right)\Phi(\zeta^{}_{t_1},\cdots,\zeta^{}_{t_n})\varphi(Q)=0.
 \]
Finally, using continuous functions $\Phi$ and $\varphi$ to approximate indicator function and by monotone class theorem we get $\mathcal M^{\phi,X^{(3)},Y^{},Q}$ is a $(\mathbb P^{m,*},(\mathcal F_t)_{0\leq t\leq T})$ martingale. 
\end{proof}
%

\vspace{-3.6mm}
For each $(\underline{\mathbb P},\underline X,\underline Q,\underline\mu)$, define

\vspace{-7.6mm}
\begin{equation*}\label{eq:mathcal J}
	\begin{split}
	\mathcal J(\underline{\mathbb P},\underline X,\underline Q,\underline\mu):=&~\mathbb E^{\underline{\mathbb P}}\left[\sum_{j=1}^d\int_{\underline X^{(1)}_{j,0-}}^{\underline X^{(1)}_{j,T+1}}h_j(x)\,dx-\int_0^Th(\underline X^{(1)}_t)\cdot b^{(1)}(t,\underline X^{(1)}_t,\underline\mu_t^{(1)})\,dt-\frac{1}{2}\sum_{j=1}^d\int_0^Ta^{(1)}_{jj}(t)h'_j(\underline X^{(1)}_{j,t})\,dt\right.\\
	&~+\sum_{j=1}^d\int_{\underline X^{(2)}_{j,0-}}^{\underline X^{(2)}_{j,T+1}}h_j(x)\,dx-\int_0^Th(\underline X^{(2)}_t)\cdot b^{(2)}(t,\underline X^{(2)}_t,\underline\mu_t^{(2)})\,dt-\frac{1}{2}\sum_{j=1}^d\int_0^Ta^{(2)}_{jj}(t)h'_j(\underline X^{(2)}_{j,t})\,dt\\
	&~\left.+\int_0^T\int_Uf(t,\underline X_t,\underline\mu_t,u)\,\underline Q_t(du)dt+g( \underline X_{T+1},\underline\mu_{T+1})\right].
	\end{split}
\end{equation*}
By Lemma \ref{lem-new-cost} and \eqref{trivial-extension} it holds that

\vspace{-7.6mm}
\begin{equation}\label{cost-Jo-mathcalJ}
J^o(\mathbb P;\nu)=\mathcal J(\mathbb P,X^o,Q,\nu)\quad \textrm{and}\quad J(\mathbb P;\mu)=\mathcal J(\mathbb P,X,Q,\mu)\textrm{ if }\mu_{T+1}=\mu_T.
\end{equation}
With \eqref{cost-Jo-mathcalJ}, the next Lemma shows any admissible relaxed control corresponding to $\mu^{m,*}$ inducing a finite cost can be approximated by a sequence of admissible relaxed controls corresponding to $\mu^{m,[k],*}$.
\begin{lemma}\label{app-convergence-any-law}
For any $\mathbb P\in\mathcal R^m(\mu^{m,*})$ with $J(\mathbb P;\mu^{m,*})<\infty$, we can find a sequence $\mathbb P^k\in\mathcal R^{m,[k]}(\mu^{m,[k],*})$ such that
$
    J^o(\mathbb P^k;\mu^{m,[k],*})\rightarrow J(\mathbb P;\mu^{m,*}),
$
where $\mu^{m,*}$ is defined in \eqref{def-mu-*}.
\end{lemma}

\vspace{-5.2mm}
\textit{Proof.}
The admissibility of $\mathbb P$ implies the existence of $(\breve\Omega,\breve{\mathcal F},\breve{\mathbb P})$ and $(\breve X^{},\breve Q,\breve Z,\breve W,\breve N)$ such that

\vspace{-5.3mm}
\begin{equation}\label{step2-breve-X1}
    \breve X^{(i)}_\cdot=\int_0^\cdot b^{(i)}(s,\breve X^{(i)}_s,\mu^{(i),m,*}_s)\,ds+\kappa^{(i)}\overline\mu^{(i),m,*}_\cdot+\eta^{(i)}\breve Z^{(i)}_\cdot+\int_0^\cdot\sigma^{(i)}(s)\,d\breve W^{(i)}_s,\quad i=1,2,
\end{equation}

\vspace{-5.6mm}
\begin{equation}\label{step2-breve-X3}
        \breve{X}^{(3)}_\cdot=\int_0^\cdot\int_U b^{(3)}(s,\breve{X}^{(3)}_s,u)\,\breve{Q}_s(du)ds+\int_0^\cdot\int_Ul(s,u)\widetilde{\breve N}(ds,du)+\alpha^{(1)}\breve{Z}^{(1)}_\cdot-\alpha^{(2)}\breve Z^{(2)}_\cdot,
\end{equation}
and
\begin{equation}\label{app-identity-law-2}
    \breve{\mathbb P}\circ\left(\breve X,\breve Q,\breve Z\right)^{-1}=\mathbb P^{}\circ \left(X,Q,Z\right)^{-1}.
\end{equation}
Let $\breve{X}^{(1),k}$, $\breve{X}^{(2),k}$ and $\breve{X}^{(3),k}$ be the unique strong solution to the following dynamics, respectively,
\begin{equation}\label{step2-breve-X2k}
    \breve{X}^{(i),k}_\cdot=\int_0^\cdot b^{(i)}(s,\breve X^{(i),k}_s,\mu^{(i),m,[k],*}_s)\,ds+\kappa^{(i)}\overline\mu^{(i),m,[k],*}_\cdot+\eta^{(i)}\breve Z^{(i),[k]}_\cdot+\int_0^\cdot\sigma^{(i)}(s)\breve W^{(i)}_s,\quad i=1,2
\end{equation}
and
\begin{equation}\label{step2-breve-X3k}
    \breve{X}^{(3),k}_\cdot=\int_0^\cdot\int_U b^{(3)}(s,\breve{X}^{(3),k}_s,u)\,\breve{Q}_s(du)ds+\int_0^\cdot\int_Ul(s,u)\widetilde{\breve N}(ds,du)+\alpha^{(1)}\breve{Z}^{(1),[k]}_\cdot-\alpha^{(2)}\breve Z^{(2),[k]}_\cdot,
\end{equation}
where we recall $\mu^{m,[k],*}$ is the optimal mean field aggregation from Section \ref{section-finite-fuel-smooth}. 
Define

\vspace{-5.6mm}
\begin{equation}\label{app-identity-law-1}
    \mathbb P^k\circ(X^o,Q,Z)^{-1}:=\breve{\mathbb P}\circ\left(\breve X^{k},\breve Q,\breve Z\right)^{-1},
\end{equation}

\vspace{-2.6mm}
which together with the equations \eqref{step2-breve-X2k}-\eqref{step2-breve-X3k} implies the admissibility of $\mathbb P^k$, i.e., $\mathbb P^k\in\mathcal R^{m,[k]}(\mu^{m,[k],*})$. By the fact
$\breve Z^{(i),[k]}\rightarrow\breve{Z}^{(i)}\textrm{ in }(\widetilde{\mathcal D}_{0,T+1},M_1)$  $\breve{\mathbb P}$\textrm{ a.s.,}
we have the following convergence from \eqref{step2-breve-X1}-\eqref{step2-breve-X3k}
\begin{equation}\label{app-convergence-breveXk-breveX}
    \mathbb E^{\breve{\mathbb P}}\left[\int_0^{T+1}\left\|\breve X^k_t-\breve X_t\right\|^p\,dt\right]\rightarrow 0\textrm{ and }\breve X^{k}_{T+1}\rightarrow \breve X^{}_{T+1}\overset{\textrm{by }\eqref{trivial-extension}}{=}\breve X_T.
\end{equation}
Moreover, by \eqref{step2-a.s.-1}-\eqref{def-mu-*} we have

\vspace{-6.6mm}
\begin{equation}\label{eq:mu-T+1=mu_T}
	\mu^{m,[k],*}_t\rightarrow\mu^{m,*}_t\textrm{ for almost all }t\in[0,T+1]\textrm{ including }T+1\textrm{ and }\mu^{m,*}_{T+1}=\mu^{m,*}_T.
\end{equation}

\vspace{-4.6mm}
Therefore, by choosing a subsequence if necessary, we have

\vspace{-9.6mm}
\begin{align*}
        &~J^o(\mathbb P^k;\mu^{m,[k],*})\overset{\textrm{ by }\eqref{cost-Jo-mathcalJ}}{
       =}\mathcal J(\mathbb P^k,X^o,Q,\mu^{m,[k],*})
    \overset{\textrm{by }\eqref{app-identity-law-1}}{=}\mathcal J(\check{\mathbb P},\check X^k,\check Q,\mu^{m,[k],*})\\
    \overset{\textrm{by }\eqref{app-convergence-breveXk-breveX},\eqref{eq:mu-T+1=mu_T}}{\rightarrow}&~\mathcal J(\check{\mathbb P},\check X,\check Q,\mu^{m,*})
    \overset{\textrm{by }\eqref{app-identity-law-2}}{=}\mathcal J(\mathbb P,X,Q,\mu^{m,*})
    \overset{\textrm{by }\eqref{cost-Jo-mathcalJ}}{=}J(\mathbb P;\mu^{m,*}).\qquad\qquad \qquad\qquad\quad\qquad \square
\end{align*}

\vspace{-2.6mm}
The following theorem shows $\mu^{m,*}$ defined in \eqref{def-mu-*} is an equilibrium of \eqref{general-MFG}.

\vspace{-1.6mm}
\begin{theorem}\label{thm:existence-m}
	Under Assumptions $\mathcal A_1$-$\mathcal A_6$, it holds that $\mathbb P^{m,*}\in\mathcal R^{m,*}(\mu^{m,*})$.
\end{theorem}

\vspace{-4.6mm}
\textit{Proof.} For any $\mathbb P\in\mathcal R^m(\mu^{m,*})$ with $J(\mathbb P;\mu^{m,*})<\infty$, let $\mathbb P^k$ be the probability measure constructed in Lemma \ref{app-convergence-any-law}. Thus, Lemma \ref{app-convergence-any-law} and optimality of $\mathbb P^{m,[k],*}$ imply

\vspace{-9.6mm}
\begin{align*}
        &~J(\mathbb P;\mu^{m,*})\overset{\textrm{by Lemma } \ref{app-convergence-any-law}}{\leftarrow} J^o(\mathbb P^k;\mu^{m,[k],*})\geq J^o(\mathbb P^{m,[k],*};\mu^{m,[k],*})
        =\mathcal J(\mathbb P^{m,[k],*},X^o,Q,\mu^{m,[k],*})\\
   &~ \overset{\textrm{by }\eqref{app-identity-law(3)}}{=}\mathcal J(\mathring{\mathbb Q},\mathring X^k,\mathring Q^k,\mu^{m,[k],*})
    \overset{\textrm{by }\eqref{step2-a.s.-1}\textrm{ and }\eqref{step2-convergence-bar-X3k-to-bar-X-3}}{\rightarrow}\mathcal J(\mathring{\mathbb Q},\mathring X,\mathring Q,\mu^{m,*})
    \overset{\textrm{by }\eqref{def-P-*},\eqref{cost-Jo-mathcalJ}}{=}J(\mathbb P^{m,*};\mu^{m,*}). ~\qquad~ \square
\end{align*}

\vspace{-10.6mm}
\section{Existence of Equilibria with General Singular Controls}\label{sec:general-control}

\vspace{-3.9mm}
In this section, we prove Theorem \ref{thm:general-control}. From Section \ref{proof}, for each $m\in(0,\infty)$ there exists an equilibrium $\mathbb P^{m,*}$ for the MFG with finite fuel constraint, i.e., $\mathbb P^{m,*}\in\mathcal R^{m,*}(\mu^{m,*})$, where $\mu^{m,*}=(\mathbb P^{m,*}\circ(Z^{(1)})^{-1},\mathbb P^{m,*}\circ(Z^{(2)})^{-1},\mathbb P^{m,*}\circ(X^{(1)})^{-1},\mathbb P^{m,*}\circ(X^{(2)})^{-1},\mathbb P^{m,*}\circ(X^{(3)})^{-1})$. In order to drop the finite fuel constraint, we need the following uniform bound for $p$ moments of $Z^{(1)}$ and $Z^{(2)}$ under $\mathbb P^{m,*}$. 

\vspace{-2.3mm}
\begin{lemma}\label{lem:uniform-bound}
Under assumptions $\mathcal A_1$, $\mathcal A_2$, $\mathcal A_4$, $\mathcal A_6$ and $\mathcal A_7$,  we have the following uniform estimate
$
	\sup_m\mathbb E^{\mathbb P^{m,*}}[\|Z^{(1)}_T\|^{p}+\|Z^{(2)}_T\|^{p}]<\infty.
$
\end{lemma}

\vspace{-6.6mm}
\begin{proof}
	W.l.o.g., we assume $\eta^{(i)}\alpha^{(i)}\neq 0$ for $i=1,2$ and $p$ is an integer.
Proposition \ref{martingale-representation}	yields $\mathbb Q$ and $(\ddot X,\ddot Z,\ddot Q,\ddot W,\ddot N)$ such that $\mathbb P^{m,*}=\mathbb Q\circ (\ddot X,\ddot Z,\ddot Q)^{-1}$ and

\vspace{-10.1mm}
\begin{equation*}
	\begin{split}
		\ddot X^{(i)}_t=&~\int_0^t b^{(i)}(s,\ddot X^{(i)}_s,\mu^{(i),m,*}_s)\,ds+\kappa^{(i)}\overline\mu^{(i),m,*}_t+\eta^{(i)} \ddot Z^{(i)}_t+\int_0^t\sigma^{(i)}(s)\,d\ddot W^{(i)}_s,\quad i=1,2,\\
		\ddot X^{(3)}_t=&~\int_0^t\int_U b^{(3)}(s,\ddot X^{(3)}_s,u)\ddot Q(ds,du)+\alpha^{(1)}\ddot Z^{(1)}_t-\alpha^{(2)}\ddot Z^{(2)}_t+ \int_0^t\int_U l(s,u)\widetilde{\ddot N}(ds,du).
	\end{split}
\end{equation*}

\vspace{-5.6mm}
In the following, we first establish lower bounds of $\ddot X^{(i)}$ and $J(\mathbb P^{m,*};\mu^{m,*})$, then construct a $\mathbb P_0\in\mathcal R^m(\mu^{m,*})$ and establish an upper bound of $J(\mathbb P_0;\mu^{m,*})$, finally we complete the proof by using the optimality of $\mathbb P^{m,*}$.

\vspace{-2.6mm}
\textbf{Step 1: lower bound of $\ddot X^{(i)}$ and $J(\mathbb P^{m,*};\mu^{m,*})$.}

\vspace{-3.6mm}
By assumption $\mathcal A_7$, it holds for each $i=1,2$ and $j=1,\cdots,d$ that

\vspace{-9.6mm}
\begin{equation*}
	\begin{split}
	|\ddot X^{(i)}_{j,t}|
	\geq &~\left|   	\int_0^t \Big(b^{(i)}_j(s,\ddot X^{(i)}_s,\mu^{(i),m,*}_s)+C_4\Big)\,ds+\kappa^{(i)}\overline\mu^{(i),m,*}_{j,t}+\eta^{(i)} \ddot Z^{(i)}_{j,t}	\right|-C_4t-\left|\left(\int_0^t\sigma^{(i)}(s)\,d\ddot W^{(i)}_s\right)_j			\right|		\\
	\geq&~   	\int_0^t C_4|\ddot X^{(i)}_{j,s}|\,ds+\int_0^tC_4\mathcal W_p(\mu^{(i),m,*}_s,\delta_0	)\,ds+\kappa^{(i)}\overline\mu^{(i),m,*}_{j,t}+\eta^{(i)} \ddot Z^{(i)}_{j,t}-C_4t-\left\|\int_0^t\sigma^{(i)}(s)\,d\ddot W^{(i)}_s			\right\|,
	\end{split}
\end{equation*}

\vspace{-4.1mm}
which implies
$		-|\ddot X^{(i)}_{j,t}|\leq \int_0^t C_4(-|\ddot X^{(i)}_{j,s}|)\,ds-\int_0^tC_4\mathcal W_p(\mu^{(i),m,*}_s,\delta_0	)\,ds  - \kappa^{(i)}\overline\mu^{(i),m,*}_{j,t}-\eta^{(i)} \ddot Z^{(i)}_{j,t}+C_4t+\left\|\int_0^t\sigma^{(i)}(s)\,d\ddot W^{(i)}_s			\right\|.$
Here, we recall $\left\|\int_0^t\sigma^{(i)}(s)\,d\ddot W^{(i)}_s			\right\|=\max_{j=1,\cdots,d}\left|\left(\int_0^t\sigma^{(i)}(s)\,d\ddot W^{(i)}_s\right)_j			\right|$.
Gr\"onwall's inequality implies that

\vspace{-9.6mm}
\begin{equation*}
	\begin{split}
		-|\ddot X^{(i)}_{j,t}|\leq&~-\int_0^tC_4\mathcal W_p(\mu^{(i),m,*}_s,\delta_0	)\,ds  - \kappa^{(i)}\overline\mu^{(i),m,*}_{j,t}-\eta^{(i)} \ddot Z^{(i)}_{j,t}+C_4t+\left\|\int_0^t\sigma^{(i)}(s)\,d\ddot W^{(i)}_s			\right\| \\
		&~+\int_0^tC_4\left(  -\int_0^sC_4\mathcal W_p(\mu^{(i),m,*}_r,\delta_0	)\,dr		    - \kappa^{(i)}\overline\mu^{(i),m,*}_{j,s}-\eta^{(i)} \ddot Z^{(i)}_{j,s}+C_4s+\left\|\int_0^s\sigma^{(i)}(r)\,d\ddot W^{(i)}_r			\right\|  \right) e^{C_4(t-s)}\,ds,
	\end{split}
\end{equation*}
which further implies that

\vspace{-9.1mm}
\begin{align}\label{eq:lower-bound-j-1}
		|\ddot X^{(i)}_{j,t}|
		\geq&~ \int_0^tC_4\mathcal W_p(\mu^{(i),m,*}_s,\delta_0	)\,ds+	\kappa^{(i)}\overline\mu^{(i),m,*}_{j,t}+\int_0^tC_4\bigg(\int_0^sC_4\mathcal W_p(\mu^{(i),m,*}_r,\delta_0	)\,dr+ \kappa^{(i)}\overline\mu^{(i),m,*}_{j,s}\bigg)e^{C_4(t-s)}\,ds  \nonumber\\
		&~+\left(\eta^{(i)} \ddot Z^{(i)}_{j,t}+\int_0^tC_4\eta^{(i)} \ddot Z^{(i)}_{j,s} e^{C_4(t-s)}\,ds\right)   \nonumber\\
		&~-C_4t-\left\|\int_0^t\sigma^{(i)}(s)\,d\ddot W^{(i)}_s			\right\|-C_4\int_0^t\left(C_4s+\left\|\int_0^s\sigma^{(i)}(r)\,d\ddot W^{(i)}_r			\right\|  \right) e^{C_4(t-s)}\,ds   \nonumber\\
		:=&~I^{(i)}_{1,j}(t)+I^{(i)}_{2,j}(t)-I_3^{(i)}(t).
\end{align}

\vspace{-4.6mm}
Define $\Omega_1=\left\{ \omega:  I^{(i)}_{3}(t)>I^{(i)}_{1,j}(t) + I^{(i)}_{2,j}(t) \right\}$ and $\Omega_2=\left\{\omega:	I^{(i)}_{1,j}(t)+ I^{(i)}_{2,j}(t) \geq I^{(i)}_{3}(t)	\right\}$. On $\Omega_2$, the r.h.s. of \eqref{eq:lower-bound-j-1} is nonnegative. Taking $|\cdot|^p$ on both sides, we get

\vspace{-9.3mm}
\begin{equation}\label{eq:lower-bound-j-2}
	\begin{split}
	|\ddot X^{(i)}_{j,t}|^p		\geq&~ \{I^{(i)}_{1,j}(t)+I^{(i)}_{2,j}(t)\}^p1_{\Omega_2}+(-I^{(i)}_{3}(t))^p1_{\Omega_2}+\sum_{n=1}^{p-1} \begin{pmatrix}
				p \\
				n
			\end{pmatrix} \{I^{(i)}_{1,j}(t)+I^{(i)}_{2,j}(t)\}^{p-n}(-I^{(i)}_{3}(t))^n1_{\Omega_2}.
	\end{split}
\end{equation}

\vspace{-5.6mm}
On $\Omega_1$ it holds that $I^{(i)}_{3}(t)^p	1_{\Omega_1}\geq \{I^{(i)}_{1,j}(t)+I^{(i)}_{2,j}(t)	\} ^p1_{\Omega_1}$, which together with \eqref{eq:lower-bound-j-2} implies 
	$	|\ddot X^{(i)}_{j,t}|^p + I^{(i)}_{3}(t)^p	1_{\Omega_1} \geq \{	I^{(i)}_{1,j}(t)+I^{(i)}_{2,j}(t)	\}^p + (-I^{(i)}_{3}(t))^p1_{\Omega_2}+\sum_{n=1}^{p-1} \begin{pmatrix}
			p \\
			n
		\end{pmatrix} \{I^{(i)}_{1,j}(t)+I^{(i)}_{2,j}(t)\}^{p-n}(-I_{3}(t))^n1_{\Omega_2}$,
which further implies that by moving $I_3^{(i)}(t)^p1_{\Omega_1}$ to the r.h.s.

\vspace{-8.6mm}
\begin{equation}\label{eq:lower-bound-j-3}
	\begin{split}
		|\ddot X^{(i)}_{j,t}|^p  
	\geq&~	I^{(i)}_{1,j}(t)^p+I^{(i)}_{2,j}(t)^p-I^{(i)}_{3}(t)^p-\sum_{n=1}^{p-1} \begin{pmatrix}
		p \\
		n
	\end{pmatrix} \{I^{(i)}_{1,j}(t)+I^{(i)}_{2,j}(t)	\}^{p-n}
	I^{(i)}_{3}(t)^n.
	\end{split}
\end{equation}

\vspace{-5.1mm}
By taking maximum over $j=1,\cdots,d$, one has

\vspace{-8.6mm}
\begin{equation}\label{eq:lower-bound-max}
	\begin{split}
		\|\ddot X^{(i)}_t\|^p\geq&~  	\|I^{(i)}_{1}(t)\|^p-I^{(i)}_3(t)^p-\sum_{n=1}^{p-1} \begin{pmatrix}
			p \\
			n
		\end{pmatrix} \left\{	\|I^{(i)}_{1}(t)\|+\|I^{(i)}_{2}(t)\|	\right\}^{p-n}I^{(i)}_3(t)^n,
	\end{split}
\end{equation}

\vspace{-5.6mm}
where $\|I^{(i)}_1(t)\|=\max_{j=1,\cdots,d}I^{(i)}_{1,j}(t)$. 
By assumptions $\mathcal A_3$, $\mathcal A_4$ and $\mathcal A_7$,
\eqref{eq:lower-bound-j-3} and \eqref{eq:lower-bound-max}, it implies that

\vspace{-9.6mm}
\begin{align}\label{eq:lower-bound-J}
		&~J(\mathbb P^{m,*};\mu^{m,*})  \nonumber\\
		\geq &~ -2C_4d-3C_4T-C_4-\frac{C_2}{2}\sum_{i=1}^2\sum_{j=1}^d\mathbb E^{\mathbb P^{m,*}}\left[	\int_0^T|a_{jj}^{(i)}(t)|(1+ |X^{(i)}_{j,t}|^{p-1})\,dt		\right]  \nonumber\\
		&~+C_4\sum_{i=1}^2\sum_{j=1}^d\left(	I^{(i)}_{1,j}(T)^p+\mathbb E^{\mathbb P^{m,*}} \left[	  I^{(i)}_{2,j}(T)^p-I_3^{(i)}(T)^p -\sum_{n=1}^{p-1}	  \begin{pmatrix}
			p \nonumber\\
			n
		\end{pmatrix} \left\{	I^{(i)}_{1,j}(t)+I^{(i)}_{2,j}(t)	\right\}^{p-n}I^{(i)}_3(t)^n		\right]\right)    \nonumber\\
		&~+C_4\sum_{i=1}^2\left(	\int_0^T\|I^{(i)}_{1}(t)\|^p\,dt -\mathbb E^{\mathbb P^{m,*}}\left[	\int_0^T	I^{(i)}_3(t)^p+\sum_{n=1}^{p-1} \begin{pmatrix}
			p \\
			n
		\end{pmatrix} \left\{	\|I^{(i)}_{1}(t)\|+\|I^{(i)}_{2}(t)\|	\right\}^{p-n}I^{(i)}_3(t)^n		\,dt		\right]		\right)   \nonumber\\
		&~+C_4\left(	\int_0^T\max_{i=1,2}\|I_1^{(i)}(t)\|^p\,dt	-\mathbb E^{\mathbb P^{m,*}}\left[	\int_0^T \max_{i=1,2}\left\{	I^{(i)}_3(t)^p+\sum_{n=1}^{p-1} \begin{pmatrix}
			p \\
			n
		\end{pmatrix} \left\{	\|I^{(i)}_{1}(t)\|+\|I^{(i)}_{2}(t)\|	\right\}^{p-n}I^{(i)}_3(t)^n			\right\}   \,dt			\right]			\right)		\nonumber\\
	&~+C_4\left(	\max_{i=1,2}\|I^{(i)}_{1}(T)\|^p-\mathbb E^{\mathbb P^{m,*}}\left[\max_{i=1,2}\left\{I^{(i)}_3(T)^p+\sum_{n=1}^{p-1} \begin{pmatrix}
		p \\
		n
	\end{pmatrix} \left\{	\|I^{(i)}_{1}(T)\|+\|I^{(i)}_{2}(T)\|	\right\}^{p-n}I^{(i)}_3(T)^n\right\}\right]		\right)  \nonumber\\
	&~+C_4\sum_{i=1}^2\int_0^T\mathcal W_p^p(	\mu^{(i),m,*}_t,\delta_0	)		\,dt+C_4\int_0^T\mathcal W_p^p(\mu^{m,*}_t,\delta_0)\,dt + C_4\mathcal W_p^p(\mu^{m,*}_T,\delta_0).
\end{align}

\vspace{-4.1mm}
\textbf{Step 2: construction of $\mathbb P_0$ and upper bound of $J(\mathbb P_0;\mu^{m,*})$.}

\vspace{-2.6mm}
Choose $u_0\in U$ and a probability measure $\mathbb P$ on some probability space that is large enough to support two Brownian motions $\widehat W^{(1)}$ and $\widehat W^{(2)}$ and a Poisson process $\widehat N$ with intensity $\lambda$. Define $\breve X^{(i)}$, $i=1,2,3$ as the unique strong solutions to the following SDEs:
%
		$\breve X^{(i)}_t=\int_0^t b(s,\breve X^{(i)}_s,\mu^{(i),m,*}_s)\,ds+\kappa^{(i)}\overline\mu^{(i),m,*}_t+\int_0^t\sigma^{(i)}_s\,d\widehat W^{(i)}_s,~ i=1,2,$ and 
		$\breve X^{(3)}_t=\int_0^tb^{(3)}(s,\breve X^{(3)}_s,u_0)\,ds+\int_0^t l(s,u_0)\widetilde{\widehat N}(ds).$
%
Define $\mathbb P_0:=\mathbb P\circ (\breve X,\breve Z,\breve Q)^{-1}$, where $\breve Z\equiv 0$ and $\breve Q(dt,du)\equiv\delta_{u_0}(du)dt$. Then $\mathbb P_0\in\mathcal R^m(\mu^{m,*})$ and $\mathbb P_0(Z= 0,Q=\delta_{u_0}(du)dt)=1$. 
%
By assumption $\mathcal A_7$, we have
$		|\breve X^{(i)}_{j,t}|		\leq 	\int_0^t C_4\left( 1+|\breve X^{(i)}_{j,t}| +\mathcal W_p(\mu^{(i),m,*}_s,\delta_0)		\right)\,ds +\kappa^{(i)} \overline \mu^{m,*}_t+\left\|	\int_0^t\sigma^{(i)}_s\,d\widehat W^{(i)}_s		\right\|.$
Gr\"onwall's inequality implies that

\vspace{-8.6mm}
\begin{equation}\label{eq:upper-bound-j}
	\begin{split}
		|\breve X_{j,t}^{(i)}|\leq&~C_4\int_0^t\mathcal W_p(\mu^{(i),m,*}_s,\delta_0)\,ds+ \kappa^{(i)} \overline \mu^{m,*}_{j,t}+C_4t+\left\|	\int_0^t\sigma^{(i)}_s\,d\widehat W^{(i)}_s		\right\| \\ &~+C_4\int_0^t\left( C_4\int_0^s\mathcal W_p(\mu^{(i),m,*}_r,\delta_0)\,dr+\kappa^{(i)} \overline \mu^{m,*}_{j,s}+C_4s+\left\|	\int_0^s\sigma^{(i)}_r\,d\widehat W^{(i)}_r		\right\|  \right)e^{C_4(t-s)}\,ds\\
		=&~  I^{(i)}_{1,j}(t) + I^{(i)}_3(t),\qquad i=1,2,
	\end{split}
\end{equation}
where $I_{1,j}^{(i)}$ and $I_3^{(i)}$ are defined as in \textbf{Step 1}. The inequality \eqref{eq:upper-bound-j} further implies that by taking maximum over $j=1,\cdots,d$
\begin{equation}\label{eq:upper-bound-max}
	\|\breve X^{(i)}_t\|\leq \|I^{(i)}_1(t)\| + I_3^{(i)}(t),\qquad i=1,2.
\end{equation}
Moreover, standard argument implies $\mathbb E^{\mathbb P_0}\left[\sup_{0\leq t\leq T}\|X^{(3)}_t\|^p\right]+\mathbb E^{\mathbb P_0}\left[\int_0^T\|X^{(3)}_t\|^p\,dt\right]\leq C<\infty$.
By assumptions $\mathcal A_3$, $\mathcal A_4$, $\mathcal A_7$, \eqref{eq:upper-bound-j} and \eqref{eq:upper-bound-max}, we have

\vspace{-9.6mm}
\begin{align}\label{general-control-estimate-J0}
	&~J(\mathbb P_0;\mu^{m,*})  
\leq  2C_4d+3C_4T+C_4 +\frac{C_2}{2}\sum_{i=1}^2\sum_{j=1}^d\mathbb E^{\mathbb P_0}\left[	\int_0^T |a_{jj}^{(i)}(t)| (1+|X_{j,t}^{(i)}|^{p-1})\,dt	\right]   \nonumber\\
	&~+C_4\sum_{i=1}^2\sum_{j=1}^d\left(	I_{1,j}^{(i)}(T)^p	+ \mathbb E^{\mathbb P_0}\left[I^{(i)}_3(T)^p +\sum_{n=1}^{p-1}\begin{pmatrix}	p\\ n	 \end{pmatrix} I_{1,j}^{(i)}(T)^{p-n}I_3^{(i)}(T)^n		\right] \right)    \nonumber\\
	&~+C_4 \sum_{i=1}^2\left(\int_0^T \|I_1^{(i)}(t)\|^p\,dt +\mathbb E^{\mathbb P_0}\left[\int_0^T I_3^{(i)}(t)^p	+ \sum_{n=1}^{p-1}\begin{pmatrix}	p\\ n	 \end{pmatrix} \|I_{1}^{(i)}(t)\|^{p-n}I_3^{(i)}(t)^n	\,dt\right]	\right)		\nonumber\\
	&~+C_4\left(	\int_0^T\max_{i=1,2}\|I_1^{(i)}(t)\|^p\,dt		+\mathbb E^{\mathbb P_0}\left[		\int_0^T \max_{i=1,2}\left\{I_3^{(i)}(t)^p	+ \sum_{n=1}^{p-1}\begin{pmatrix}	p\\ n	 \end{pmatrix} \|I_{1}^{(i)}(t)\|^{p-n}I_3^{(i)}(t)^n\right\}	\,dt		\right]	\right)			\nonumber\\
	&~+C_4\left(\max_{i=1,2} \| I_1^{(i)}(T) \|^p	 +\mathbb E^{\mathbb P_0}\left[ \max_{i=1,2}\left\{ I_3^{(i)}(T)^p +	\sum_{n=1}^{p-1}\begin{pmatrix}	p\\ n	 \end{pmatrix} \|I_{1}^{(i)}(T)\|^{p-n}I_3^{(i)}(T)^n \right\}	\right]		\right)      \nonumber\\
	&~+C_4\sum_{i=1}^2\int_0^T\mathcal W_p^p(	\mu^{(i),m,*}_t,\delta_0	)		\,dt+C_4\int_0^T\mathcal W_p^p(\mu^{m,*}_t,\delta_0)\,dt+C_4\mathcal W_p^p(\mu^{m,*}_T,\delta_0)+C_4C.
\end{align}

\vspace{-4.1mm}
\textbf{Step 3: complete the proof by the optimality of $\mathbb P^{m,*}$.}

\vspace{-2.1mm}
From \eqref{eq:lower-bound-J}, \eqref{general-control-estimate-J0} and $J(\mathbb P^{m,*};\mu^{m,*})\leq J(\mathbb P_0;\mu^{m,*})$, we can see that terms with $\sum_{j=1}^d|I^{(i)}_{1,j}|^p$, $\int_0^T\|I_1^{(i)}(t)\|^p\,dt$, $\|I_1^{(i)}(T)\|^p$, $\sum_{i=1}^2\int_0^T\mathcal W_p^p(	\mu^{(i),m,*}_t,\delta_0	)		\,dt$, $\int_0^T\mathcal W_p^p(\mu^{m,*}_t,\delta_0)\,dt$ and $\mathcal W_p^p(\mu^{m,*}_T,\delta_0)$ cancel out. Note that all terms with $I_3^{(i)}$ are bounded uniformly in $m$. Using H\"older's inequality to terms with $\|I_2^{(i)}\|^{p-n}(I_3^{(i)})^n$ in \eqref{eq:lower-bound-J}, we obtain a positive constant $C$ that is independent of $m$, such that

\vspace{-9.6mm}
\begin{equation*}
	\begin{split}
	\sum_{i=1}^2\mathbb E^{\mathbb P^{m,*}}[\|Z^{(i)}_T\|^p	]\leq &~ C\left(1+ \sum_{i=1}^2 \sum_{n=1}^{p-1}	\mathbb E^{\mathbb P^{m,*}}[ \|Z_T^{(i)}\|^{p-n} ]+\sum_{i=1}^2 \sum_{n=1}^{p-1}	\mathbb E^{\mathbb P^{m,*}}[ \|(Z_T^{(i)})^p\| ]^{\frac{p-n}{p}}			\right)\\
	\leq&~ C\left(1+\sum_{i=1}^2 \sum_{n=1}^{p-1}	\mathbb E^{\mathbb P^{m,*}}[ \|(Z_T^{(i)})^p\| ]^{\frac{p-n}{p}}			\right).
	\end{split}
\end{equation*}
If $\mathbb E^{\mathbb P^{m,*}}[ \|Z_T^{(i)}\|^p ]\geq (4Cp)^{\frac{p}{n}}$, it holds that $\mathbb E^{\mathbb P^{m,*}}[ \|Z^{(i)}_T\|^p ]^{\frac{p-n}{p} }\leq \frac{1}{4Cp}\mathbb E^{\mathbb P^{m,*}}[ \|Z^{(i)}_T\|^p ]$, which implies that $\mathbb E^{\mathbb P^{m,*}}[ \|Z^{(i)}_T\|^p ]^{\frac{p-n}{p} }$$\leq  (4Cp)^{\frac{p-n}{n}} + \frac{1}{4Cp}\mathbb E^{\mathbb P^{m,*}}[ \|Z^{(i)}_T\|^p ]$. Thus, we have that
$\sum_{i=1}^2\mathbb E^{\mathbb P^{m,*}}[\|Z^{(i)}_T\|^{p}]\leq C\Big(1+2\sum_{n=1}^{p-1}(4Cp)^{\frac{p-n}{n}}\Big)+\frac{1}{4}\sum_{i=1}^2	\mathbb E^{\mathbb P^{m,*}}[\|Z^{(i)}_T\|^{p}],
$
which implies the desired result.
\end{proof}

\vspace{-3.6mm}
Recall
$
	Y=X^{(3)}-\alpha^{(1)}Z^{(1)}+\alpha^{(2)}Z^{(2)}.
$
By Lemma \ref{lem:uniform-bound} and the same arguments as Lemma \ref{relative-compactness-union-control-rule}, the sequence $\{\mathbb P^{m,*}\circ(X^{(1)},X^{(2)},Q,Z^{(1)},Z^{(2)},Y)^{-1}\}_m$ is relatively compact in $\mathcal W_{p'}$ for any $1<p'<p$. Denote by $\breve{\mathbb P}^*$ the weak limit. Skorokhod representation implies  the existence of $(\breve\Omega,\breve{\mathcal F},{\mathbb Q})$ and two tuples of stochastic processes $(\breve X^{(1)},\breve X^{(2)},\breve Q,\breve Z^{(1)},\breve Z^{(2)},\breve Y)$  and $(X^{(1),m},X^{(2),m},Q^m,Z^{(1),m},Z^{(2),m},Y^m)$ such that

\vspace{-7.6mm}
\begin{equation}\label{general-control-Xm-to-Xbar}
	\left\{\begin{split}
		&~\mathbb Q\circ(X^{(1),m},X^{(2),m},Q^m,Z^{(1),m},Z^{(2),m},Y^m)^{-1}	=\mathbb P^{m,*}\circ(X^{(1)},X^{(2)},Q,Z^{(1)},Z^{(2)},Y)^{-1},\\
		&~\mathbb Q\circ(\breve X^{(1)},\breve X^{(2)},\breve Q,\breve Z^{(1)},\breve Z^{(2)},\breve Y)^{-1}=\breve{\mathbb P}^*,\\
		&~(X^{(1),m},X^{(2),m},Q^m,Z^{(1),m},Z^{(2),m},Y^m)\rightarrow (\breve X^{(1)},\breve X^{(2)},\breve Q,\breve Z^{(1)},\breve Z^{(2)},\breve Y)\quad\mathbb Q\textrm{ a.s.}.
	\end{split}\right.
\end{equation}

\vspace{-4.6mm}
Let $ X^{(3),m}=Y^m+\alpha^{(1)}Z^{(1),m}-\alpha^{(2)}Z^{(2),m}$ and $\breve X^{(3)}=\breve Y+\alpha^{(1)}\breve Z^{(1)}-\alpha^{(2)}\breve Z^{(2)}$. Consequently, we have
$
	\mathbb P^{m,*}=\mathbb Q\circ(X^{(1),m},X^{(2),m},X^{(3),m},Q^m,Z^{(1),m},Z^{(2),m})^{-1}.
$
Define the candidate of the equilibrium as

\vspace{-5.6mm}
\begin{equation}\label{general-control-candidate-equilibrium}
	\mathbb P^*=\mathbb Q\circ (\breve X,\breve Q,\breve Z)^{-1},~\mu^{(i),*}=\mathbb P^*\circ(Z^{(i)})^{-1}, ~i=1,2,~ \mu^{(j),*}=\mathbb P^*\circ (X^{(j)})^{-1},~j=1,2,3.
\end{equation}

\vspace{-3.6mm}
For each constant $K$ define $J_K(\mathbb P;\mu)$ the same as $J(\mathbb P;\mu)$ in Lemma \ref{lem-new-cost} but with $h_j$, $h\cdot b^{(i)}$, $a^{(i)}_{jj}h'_j$, $f$ and $g$ replaced by $h_j\wedge K$, $(h\cdot b^{(i)})\wedge K$, $(a^{(i)}_{jj}h'_j)\wedge K$, $f\wedge K$ and $g\wedge K$. Then from Lemma \ref{lem-new-cost}, \eqref{general-control-Xm-to-Xbar} and \eqref{ass:convergence-lower-order} in assumption $\mathcal A_7$ we have
$
\lim_{m\rightarrow\infty}J(\mathbb P^{m,*};\mu^{m,*})\geq \lim_{m\rightarrow\infty}J_K(\mathbb P^{m,*};\mu^{m,*}) =J_K(\mathbb P^*;\mu^{*}).
$
Letting $K$ go to infinity, monotone convergence and \eqref{ass:coercive} in assumption $\mathcal A_7$ imply

\vspace{-5.6mm}
\begin{equation}\label{general-control-lsc-J}
	\lim_{m\rightarrow\infty}J(\mathbb P^{m,*};\mu^{m,*})\geq J(\mathbb P^*;\mu^{*}).
\end{equation}

\vspace{-3.6mm}
The same argument as in Lemma \ref{lem:uniform-bound} implies $\mathbb E^{\mathbb P^*}[(Z^{(i)}_T)^p]<\infty$, which further implies $\mathbb P^*\in\mathcal P_p(\Omega)$. It can be checked by the same arguments as in Proposition \ref{admissibility-limit-P-proposition} and Lemma \ref{lem:stability-martingale-2} that $\mathbb P^*\in\mathcal R^\infty(\mu^*)$. 
Together with \eqref{general-control-lsc-J}, the next theorem concludes this section.
\begin{theorem}
	Under $\mathcal A_1-\mathcal A_7$, the probability measures defined in \eqref{general-control-candidate-equilibrium} is a relaxed solution to MFG \eqref{general-MFG} with general singular controls, i.e., $J(\mathbb P^*;\mu^*)=\sup_{\mathbb P\in\mathcal R^\infty(\mu^*)}J(\mathbb P;\mu^*)$.
\end{theorem}

\vspace{-7.6mm}
\begin{proof}
First, we prove for each $\mathbb P\in\mathcal R^\infty(\mu^*)$ with $J(\mathbb P;\mu^*)<\infty$, there exists $\mathbb P^m\in\mathcal R^m(\mu^{m,*})$ such that 

\vspace{-6.6mm}
\begin{equation}\label{general-control-J(Pm)-J(P)}
	\lim_{m\rightarrow\infty}J(\mathbb P^m;\mu^{m,*})=J(\mathbb P;\mu^*).
\end{equation}

\vspace{-4.1mm}
Note that $\mathbb P\in\mathcal R^\infty(\mu^*)$ yields a tuple $(\widehat X,\widehat Q,\widehat Z,\widehat W,\widehat N)$ defined on a probability space $(\widehat{\Omega},\widehat{\mathcal F},\widehat{\mathbb P})$ such that

\vspace{-9.6mm}
\begin{equation*}
	\left\{\begin{split}
		\widehat X^{(i)}_t=&~\int_0^tb^{(i)}(s,\widehat X^{(i)}_s,\mu^{(i),*}_s)\,ds+\kappa^{(i)}\overline\mu^{(i),*}_t+\eta^{(i)}\widehat Z^{(i)}_t+\int_0^t\sigma^{(i)}(s)\,d\widehat W^{(i)}_s,\quad i=1,2,\\
		\widehat X^{(3)}_t=&~\int_0^t\int_U b^{(3)}(s,\widehat{X}^{(3)}_s,u)\,\widehat{Q}_s(du)ds+\int_0^t\int_Ul(s,u)\widetilde{\widehat N}(ds,du)+\alpha^{(1)}\widehat{Z}^{(1)}_t-\alpha^{(2)}\widehat Z^{(2)}_t,\\
		\mathbb P=&~\widehat{\mathbb P}\circ(\widehat X,\widehat Q,\widehat Z)^{-1}.
	\end{split}\right.
\end{equation*}
Define 
\begin{equation*}
	\widehat Z^{(i),m}_{j,t}=\left\{ \begin{split}
			&~\widehat Z^{(i)}_{j,t},\quad \widehat Z^{(i)}_{j,t}\leq m,\\
			&~m,\quad \widehat Z^{(i)}_{j,t}>m,
	\end{split}
	\right. \quad i=1,2,\quad j=1,\cdots,d,\quad\textrm{and}\quad \mathbb P^m=\widehat{\mathbb P}\circ(\widehat X^m,\widehat Q,\widehat Z^m)^{-1},
\end{equation*}

\vspace{-4.6mm}
where $\widehat X^m$ is defined as $\widehat X$ with $\widehat Z$ replaced by $\widehat Z^m$  and $\mu^*$ replaced by $\mu^{m,*}$.
Obviously, $\mathbb P^m\in\mathcal R^m(\mu^{m,*})$. Next we verify the convergence \eqref{general-control-J(Pm)-J(P)}. 
By definitions of $\widehat X^{(i),m}$ and $\widehat X^{(i)}$ and Gr\"onwall's inequality, for a.e. $t\in[0,T]$ including $T$ and for $i=1,2$, we have that
$
	\| \widehat X^{(i),m}_t-\widehat X^{(i)}_t \|\rightarrow 0,
$
and for any $t\in[0,T]$,
$
	\|\widehat X^{(3),m}_t-\widehat X^{(3)}_t\|\rightarrow 0.
$
Note that $\widehat Z^m_t\leq \widehat Z_t$ componentwisely. Assumption $\mathcal A_7$ implies for $i=1,2$
\begin{align*}
	&~\left|h(\widehat X^{(i),m}_t)\cdot b^{(i)}(t,\widehat X^{(i),m}_t,\mu^{(i),m,*}_t)-h(\widehat X^{(i)}_t)\cdot b^{(i)}(t,\widehat X^{(i)}_t,\mu^{(i),*}_t)\right|\\
	\leq&~ C\Big(1+\|\widehat X^{(i),m}_t\|^{p}+\|\widehat X^{(i)}_t\|^{p}+\mathcal W^p_p(\mu^{(i),m,*}_t,\delta_0)+\mathcal W_p^p(\mu^{(i)}_t,\delta_0)\Big)\\
	\leq&~C\left(\sup_m\mathbb E^{\mathbb P^{m,*}}[\|Z^{(i)}_T\|^{p}]+\mathbb E^{\mathbb P^{*}}[\|Z^{(i)}_T\|^{p}]+\|\widehat Z^{(i)}_T\|^{p}+\left\|\int_0^t\sigma^{(i)}(s)\,d\widehat W^{(i)}_s\right\|^{p}\right),
\end{align*}

\vspace{-5.6mm}
which is integrable. Indeed, the finiteness of $\sup_m\mathbb E^{\mathbb P^{m,*}}[\|Z^{(i)}_T\|^{p}]$ is given by Lemma \ref{lem:uniform-bound}. Fatou's lemma implies $\mathbb E^{\mathbb P^{*}}[\|Z^{(i)}_T\|^{p}]<\infty$. The same argument as in Lemma \ref{lem:uniform-bound} implies $\mathbb E^{\widehat{\mathbb P}}[\|Z^{(i)}_T\|^{p}]<\infty$. 
Thus, dominated convergence yields
%
		$\mathbb E^{\widehat{\mathbb P}}\bigg[\int_0^T\Big|h(\widehat X^{(i),m}_t)\cdot b^{(i)}(t,\widehat X^{(i),m}_t,\mu^{(i),m,*}_t)-h(\widehat X^{(i)}_t)\cdot b^{(i)}(t,\widehat X^{(i)}_t,\mu^{(i),*}_t)\Big|\,dt  \bigg]\rightarrow 0.$
%
Similarly, we have the convergence of other terms in the cost. 
Finally, by \eqref{general-control-lsc-J} and the optimality of $\mathbb P^{m,*}$ w.r.t. to $\mu^{m,*}$, we have
$
		J(\mathbb P^*;\mu^*)\leq \lim_{m\rightarrow\infty}J(\mathbb P^{m,*};\mu^{m,*})\leq \lim_{m\rightarrow\infty}J(\mathbb P^m;\mu^{m,*})=J(\mathbb P;\mu^*).
$
\end{proof}

\vspace{-10.6mm}
\section{Conclusion}

\vspace{-5.6mm}
We study a class of extended MFGs with singular controls by relaxed solution method. The simultaneous jumps of singular controls in different directions make it difficult to verify the tightness. In order to establish the existence of equilibria result, we smooth the singular controls to circumvent the tightness issue and then take approximation. 

\vspace{-6.6mm}
\begin{appendix}
\section{$\mathcal D([0,T];\mathbb R^d)$ is Polish under the Weak $M_1$ Topology}\label{app:complete}

\vspace{-5.6mm}
This section proves that $\mathcal D([0,T];\mathbb R^d)$ is a Polish space under the weak $M_1$ topology. Denote by $SM_1$ and $WM_1$ the strong and the weak $M_1$ topologies, respectively. Since $SM_1$ and $WM_1$ coincide in $\mathcal D([0,T];\mathbb R)$, we use $M_1$ to denote $SM_1$ and $WM_1$ in $\mathcal D([0,T];\mathbb R)$.

\begin{proposition}\label{prop:appendix}
	The space $\mathcal D([0,T];\mathbb R^d)$ is a Polish space under the weak $M_1$ topology.
\end{proposition}

\vspace{-7.6mm}
\begin{proof}
 First, it is well known that $\mathcal D([0,T];\mathbb R^d)$ is separable under $J_1$ topology; see e.g. \cite[Section 11.5]{Whitt-2002}. Thus, $\mathcal D([0,T];\mathbb R^d)$ is separable under ($S$- and $W$-) $M_1$ topology since $J_1$ is stronger than ($S$- and $W$-) $M_1$ topology. It remains to prove the topological completeness of $WM_1$. By \cite[Theorem 12.8.1]{Whitt-2002}, $\mathcal D([0,T];\mathbb R^d)$ is topologically complete under $SM_1$. In particular, this is true when $d=1$. Therefore, there is a homeomorphic mapping $f:(\mathcal D([0,T];\mathbb R),d_{S M_1}(=d_{WM_1}=d_{M_1}))\rightarrow(\mathcal D([0,T];\mathbb R),\widehat d_s)$, where $\widehat d_s$ is the complete metric on $\mathcal D([0,T];\mathbb R)$. For any Cauchy sequence $\{x_n \}\subseteq(\mathcal D([0,T];\mathbb R^d),d_{WM_1})$, i.e., $\|x_n-x_m\|_{d_{WM_1}}\rightarrow 0$, \cite[Theorem 12.5.2]{Whitt-2002} implies that $d_{M_1}(x_n^i,x^i_m)\rightarrow 0$, for any $i=1,\cdots, d$, which implies by the continuity of $f$ that $\widehat d_s(f(x^i_n),f(x^i_m))\rightarrow 0$, for any $i=1,\cdots,d$. By the completeness of $\widehat d_s$ there exists $x^i\in\mathcal D([0,T];\mathbb R)$ such that $\widehat d_s(f(x^i_n),f(x^i))\rightarrow 0$ for each $i=1,\cdots, d$, which implies by the continuity of $f^{-1}$ that $d_{M_1}(x^i_n,x^i)\rightarrow 0$. By \cite[Theorem 12.5.2]{Whitt-2002} again we have $d_{WM_1}(x_n,x)\rightarrow 0$, where $x=(x^1,\cdots, x^d)\in\mathcal D([0,T];\mathbb R^d)$.
 \end{proof}

\vspace{-8.6mm}
\section{Transformation of the Cost Functional} 
 \begin{lemma}\label{lem-new-cost}
 	
 	\vspace{-4.6mm}
 	Under assumptions $\mathcal A_1$-$\mathcal A_4$, the cost functional \eqref{cost-relaxed-control} can be rewritten as 
 	\begin{align*}
 	J(\mathbb P;\mu^{})=&~\mathbb E^{{\mathbb P}}\left[ \sum_{i=1}^2 \sum_{j=1}^d\int_{X^{(i)}_{j,0-}}^{ X^{(i)}_{j,T}}h_j(x)\,dx-\sum_{i=1}^2\int_0^Th( X^{(i)}_t)\cdot b^{(i)}(t, X^{(i)}_t,\mu^{(i)}_t)\,dt\right.\\
 &~\left.	-\frac{1}{2}\sum_{i=1}^2\sum_{j=1}^d\int_0^Ta^{(i)}_{jj}(t)h'_j(X^{(i)}_{j,t})\,dt
 	+\int_0^T\int_Uf(t, X_t,\mu^{}_t,u)\, Q_t(du)dt+g( X_{T},\mu_T)\right].
 	\end{align*}

\vspace{-5.6mm}
 \end{lemma}
\textit{Proof.} The desired result follows from using It\^o's formula as follows and then taking expectation:
 	\begin{align*}\label{cost-transform-2}
 	&~\int_{X^{(1)}_{j,0-}}^{X^{(1)}_{j,T}}h_j(x)\,dx
 	=\int_0^Th_j(X^{(1)}_{j,s}) b^{(1)}_j(s,X^{(1)}_s,\mu^{(1)}_s)\,ds+\int_0^Th_j(X^{(1)}_{j,s-})\,d\left(\kappa^{(1)}\overline\mu^{(1)}_{j,s}+\eta^{(1)}Z^{(1)}_{j,s}\right)\\
 	&~\qquad\qquad  +\frac{1}{2}\int_0^Th'_j(X^{(1)}_{j,s})a^{(1)}_{jj}(s)\,ds
 	+\sum_{0\leq t\leq T}\left(\int_{X^{(1)}_{j,t-}}^{X^{(1)}_{j,t}}h_j(x)\,dx-h_j(X^{(1)}_{j,t-})\Delta X^{(1)}_{j,t}\right)+\textrm{martingale}\\
 	=&~\int_0^Th_j(X^{(1)}_{j,s})b^{(1)}_j(s,X^{(1)}_s,\mu^{(1)}_s)\,ds+\frac{1}{2}\int_0^Th'_j(X^{(1),j}_{s})a^{(1)}_{jj}(s)\,ds\\
 	&~+\int_0^Th_j(X^{(1)}_{j,s-})\,d(\kappa^{(1)}\overline\mu^{(1)}_{j,s}+\eta^{(1)}Z^{(1)}_{j,s})^c
 	+\sum_{0\leq t\leq T}\left(\int_{0}^{\Delta X^{(1)}_{j,t}}h_j(y+X^{(1)}_{j,t-})\,dy\right)+\textrm{martingale}. \qquad\qquad\square
 	\end{align*}
\end{appendix}

\vspace{-10.6mm}

\bibliography{bib_Guanxing_thesis}

\end{document}